\documentclass[11pt,reqno]{amsart} %reqno for right equation number 
\usepackage[a4paper,margin=3cm]{geometry}     % very customizable margins. Under some (rare) circumstances, should be loaded after hyperref
\usepackage[babel]{microtype} % many good lay-out/justification effects, see:
\usepackage[dvipsnames]{xcolor} % Colors support
\usepackage[english]{babel}               % [french, frenchb, english, ]
  \usepackage{amsmath}      % loads amstext, amsbsy, amsopn but not amssymb
 \usepackage{amssymb}      % may be redundant with amsmath
\usepackage{amsthm}
\usepackage[bb=boondox]{mathalpha}%Omit bb=boondox for default libertinus bb 
\theoremstyle{plain}
\newtheorem{theorem}{Theorem}[section]
\newtheorem{proposition}[theorem]{Proposition}
\newtheorem*{proposition*}{Proposition}
\newtheorem*{theorem*}{Theorem}
\newtheorem{lemma}[theorem]{Lemma}
\newtheorem*{lemma*}{Lemma}

\theoremstyle{definition}

\newtheorem{remark}[theorem]{Remark}
\newtheorem{question}{Question}
\usepackage{mathrsfs}      % for the command \mathscr{}
\usepackage{stmaryrd} 
\usepackage{tikz} 
\usetikzlibrary{decorations.pathreplacing}
\usepackage{enumerate}    
\usepackage{mathtools} 
\usepackage{textcomp}  % provides additional text symbols, including currency signs, trademark symbols, copyright symbols, and other special characters for use in documents.
\mathtoolsset{showonlyrefs} %hides equation numbers for all equations except those that are referenced in the document. 
\numberwithin{equation}{section}  % reset equation counters at start of each "section" and prefix numbers by section number
\usepackage[T1]{fontenc}    % fontenc is oriented to output, that is, what fonts to use for printing characters. 
 \usepackage{libertinus}
\usepackage[utf8]{inputenc} % inputenc allows the user to input accented characters directly from the keyboard; 
\usepackage[symbol]{footmisc}
 \usepackage{comment}      % provide new {comment} environment: all text inside the environment is ignored.
 \usepackage{appendix}
\usepackage{graphicx}  
\usepackage{float}                      % Improved interface for floating objects ; add [H] option

\usepackage[
breaklinks   = true,        % so long urls are correctly broken across lines
pagebackref  = true,     % add page number in bibliography and link to position in document where cited
]{hyperref}

\hypersetup{
	colorlinks=true,
	pdfpagemode=UseNone,
    citecolor=OliveGreen,
    linkcolor=blue,
    urlcolor=black,
	pdfstartview=FitW
}
\def\E{\mathbf{E}}
\def\P{\mathbf{P}} 

\def\Var{\mathbf{Var}}
 
\def\T{\mathcal{T}}
\def\L{\mathcal{L}}

\newcommand{\ind}[1]{\mathbf{1}_{\{ #1 \}}}    %Indicator function 

\newcommand*{\dif}{\ensuremath{\mathop{}\!\mathrm{d}}}

\renewcommand{\hat}[1]{\widehat{#1}}

\newenvironment{TBD}{\color{gray}}{\ignorespacesafterend}

\title{Atypical Decay Rates for Atypical Heights\\in Random Recursive Trees} 

 \author{Xinxin Chen
 \and Heng Ma} 
 \address[Xinxin Chen]
 {Beijing Normal University, School of Mathematical Sciences, China   }
 \email{xinxin.chen(at)bnu(dot)edu(dot)cn}

\address[Heng Ma]
{Faculty of Data and Decision Sciences, Technion - Israel's Institute of Technology, Haifa, 32000, Israel.}
\email{hengmamath(at)gmail(dot)com}
\urladdr{\url{https://hengmamath.github.io}}

\date{\today}

\begin{document}
 
\maketitle
\begin{abstract}
  We establish the large deviation probabilities
for the height of random recursive trees,
revealing polynomial upper-tail decay and
stretched-exponential lower-tail decay.
Remarkably, the lower tail features an atypical
prefactor that grows to infinity more slowly than any $n$-fold iterated logarithm.
\end{abstract}

%\tableofcontents 

%\newpage

\section{Introduction}

Random recursive trees (RRTs) $\{ \T_n: n \geq 1\}$, also known as uniform attachment trees,  are a class of growing random trees defined by a simple recursive rule:
\begin{itemize}
  \item The tree $\T_1$ consists of a single root vertex labeled $v_1$. 
  \item Given the tree $\T_n$ on $n$ vertices $V_{[n]} \coloneq \{ v_{1},\dots,v_{n} \}$, $\T_{n+1}$ is obtained by adding a new vertex $v_{n+1}$ and attaching it to a vertex chosen uniformly at random from the $n$ existing vertices in $\T_n$.  
\end{itemize}  
These structures serve as fundamental models in computer science and probability.  
For a comprehensive review of their properties as well as their history, we refer the reader to \cite{SM94} and  \cite{Drmota09}.

\begin{figure}[htbp]
  \centering
  \includegraphics[width=\textwidth]{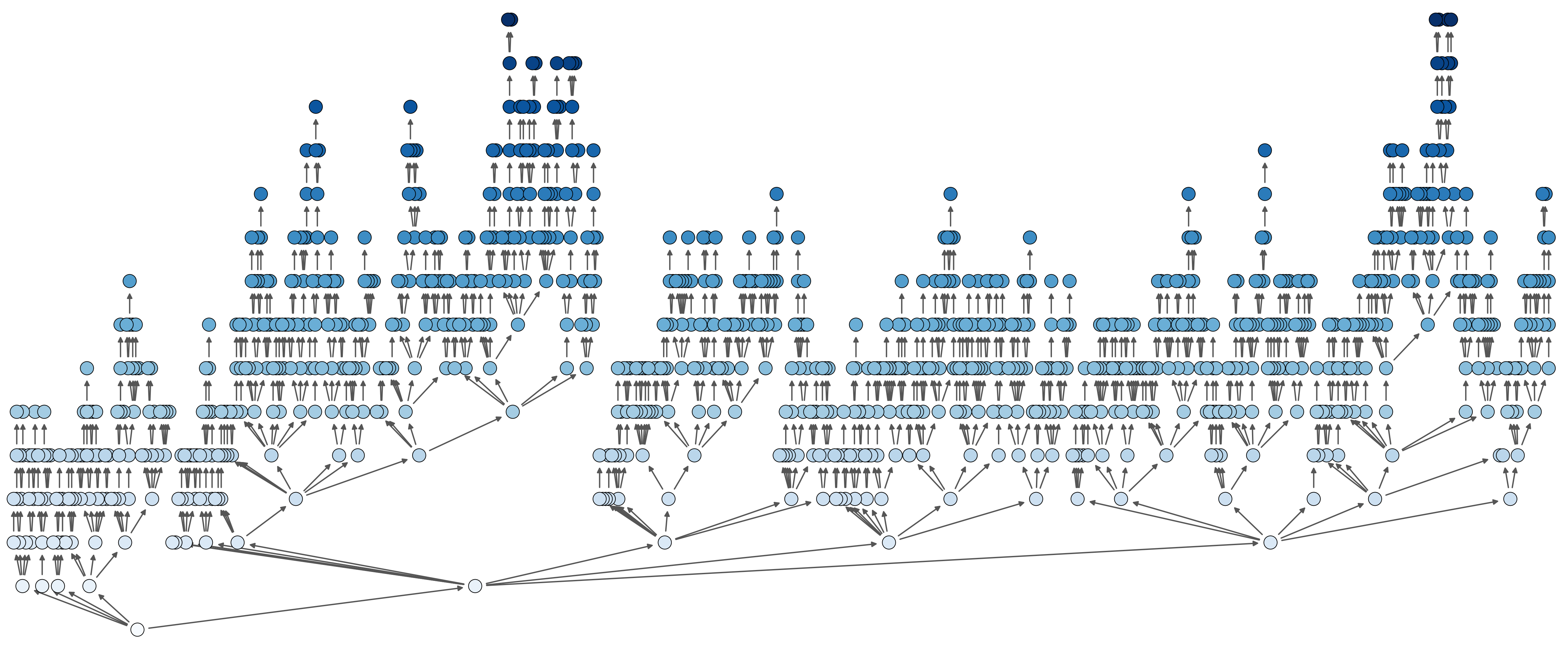}
  \caption{A realization of $\mathcal{T}_{1000}$}
  \label{fig:rrt}
\end{figure}

A fundamental statistic of interest is the height of the tree,   defined as the maximum graph distance    between the root $v_{1}$ and  any other vertex: 
\begin{equation}
    H_{n} = H(\T_{n})  \coloneq  \max_{1 \leq k \leq n } \mathrm{dist}(v_1,v_k) \quad n \geq 1.
\end{equation} 
The asymptotic behavior of $H_{n}$ has been studied extensively.  
The first-order asymptotic for the height $H_n$ was established by  Pittel \cite{Pittel94}, who showed that $H_{n} \sim e\ln n$ almost surely. 
Subsequently, Addario-Berry and Ford \cite{ABF13} provided a much finer analysis of the second- and third-order terms using connections to the minimum of a branching random walk (see also \cite{PS22, PS24} for extensions to weighted recursive trees). Their results are summarized in the following theorem:

\begin{theorem*}[{\cite[Corollary 1.3]{ABF13}}] 
The sequence of centered heights  $(H_{n}-\E[H_{n}])_{n \ge 1}$ is tight. Furthermore, as $n \to \infty$,
 \begin{equation}\label{eq-ABF-EH-bound}
  \E [H_{n} ]=e \ln n-\frac{3}{2} \ln \ln n+O(1).
 \end{equation} 
Moreover,  for any $c^{\prime}<\frac{1}{2 e}$ and all $n \geq 1$, $k \geq 1$, we have
 \begin{equation}\label{eq-ABF-bound}
 \P \left( \left|H_n-\E H_n\right| \geq k\right) \ll_{c^{\prime}} e^{-c^{\prime} k}.
 \end{equation}
\end{theorem*}
  
The deviation probability \eqref{eq-ABF-bound} was also established in \cite{Drmota09a}.
While the exponential concentration of $H_{n}$ around its mean is understood via \eqref{eq-ABF-bound}, the large-deviation behavior of the height $H_n$ remains largely unexplored. 
This leads to the  natural question:

 \begin{question}[Decay Rates]
   Fix $0< \alpha < 1< \beta $.  What are the decay rates of $\P( H_{n} < \alpha e \ln n )$ and    $ (H_{n} > \beta e\ln n)$?
 \end{question}
  
We answer this question and uncover a fundamental asymmetry between the two tails:
  While the upper tail exhibits a standard power-law decay in $n$, the lower tail exhibits a much faster stretched-exponential decay. 

  \begin{theorem}\label{thm:lowerLDP}
    Fix $0< \alpha < 1$. Define 
    \begin{equation}
     \omega_{\alpha}(n)  \coloneq   \frac{ - \ln \P \left(  H_{n} \leq \alpha e \ln n \right) }{  {n^{1- \alpha} }{(\ln n)^{-\frac{3}{2e}} }  }  .
    \end{equation} 
    Then we have 
    \begin{equation}
       \lim_{n \to \infty}  \omega_{\alpha} (n)  = \infty .
    \end{equation}
   Moreover, let $\ln^{(k)}$ denote the $k$-fold iteration of the logarithm function. Then, for every fixed $k \in \mathbb{N}$,
     \begin{equation}
      \lim_{n \to \infty} \frac{ \omega_{\alpha} (n) }{ \ln^{(k)} (n)} = 0.
    \end{equation} 
  \end{theorem}

Crucially, the correction factor $\omega_{\alpha}(n)$ involved in the lower-tail exponent is itself pathological: it diverges to infinity, yet grows more slowly than any iterated logarithm $\ln^{(k)}(n)$.  

Our second result characterizes the upper tail (unusually large height), where the decay follows a power law in $n$:

  \begin{theorem}\label{thm:upTail}
    Fix $\beta > 1$. Define $
J(\beta) \coloneq e \, \beta \ln \beta    $. Then   
     \begin{equation}
      \lim_{n \to \infty}   \frac{ - \ln \P \left(  H_{n} \geq \beta e \ln n \right) }{\ln n}=  J(\beta) .
     \end{equation} 
   \end{theorem}

\medskip

   \begin{remark}
    \label{rmk:label}
Theorems \ref{thm:lowerLDP} and \ref{thm:upTail} sharpen, and are consistent with, the concentration bound \eqref{eq-ABF-bound}. Indeed, for any fixed $\lambda>0$, \eqref{eq-ABF-bound} and \eqref{eq-ABF-EH-bound} imply
$
\P\big(H_n\ge(e+\lambda)\ln n\big)
\le n^{-\lambda/(2e)+o(1)}$.
Taking $\beta=1+\lambda/e$ in Theorem~\ref{thm:upTail}, we obtain the sharper asymptotic
$
\P\big(H_n\ge(e+\lambda)\ln n\big)
=n^{-(e+\lambda)\ln(1+\lambda/e)+o(1)}$.
This is consistent with the concentration bound since
\[
(e+\lambda)\ln(1+\lambda/e)>\frac{\lambda}{2e}.
\]
Similarly, \eqref{eq-ABF-bound} gives only a polynomial upper bound for the lower tail, whereas Theorem~\ref{thm:lowerLDP} shows that the lower-tail probability is stretched-exponentially small.
   \end{remark}

   \begin{remark}
Our results can be reinterpreted as enumerative results for increasing trees, since the random recursive tree corresponds to the uniform distribution over increasing trees. Recall that an increasing tree on $n$ nodes is a rooted tree whose nodes are labeled with distinct integers, typically from $\{1,2,\dots,n\}$, such that labels strictly increase along any path directed away from the root. The total number of such increasing trees is $(n-1)!$. Consequently, the number of increasing trees whose height is at most $\alpha e \ln n$ equals
\[
(n-1)! \, \P(H_n \le \alpha e \ln n) = \exp\!\Big( n \ln n - n - \omega_\alpha(n) n^{1-\alpha} (\ln n)^{-\frac{3}{2e}  } +O(\ln n) \Big),
\]
where $1 \ll \omega_{\alpha}(n) \ll \ln^{(k)} n$ as $n \to \infty$ for any fixed $k$.
Similarly,   the number of increasing trees whose height is at least $\beta e \ln n$ equals
\[
(n-1)! \, \P(H_n \ge \beta e \ln n) = \exp\!\Big( n \ln n - n - \Big[\frac{1}{2}+ J(\beta) \Big] \ln n - \Theta(\ln \ln n) \Big),
\]
where  the $\Theta(\ln \ln n)$   term follows from   Proposition~\ref{prop:uppLDP}. \end{remark}

\begin{remark}
There is also a shortest-path interpretation of our results. Let $K_n$ be the complete graph on $n$ vertices, assign i.i.d. exponential weights to its edges, and fix a distinguished root $o$.
Since the edge weights have a continuous distribution, the shortest path from $o$ to every vertex is almost surely unique, and these paths form the shortest-path tree rooted at $o$.
If the vertices are relabeled in the order in which they are discovered by Dijkstra's algorithm, then this shortest-path tree has the same distribution as the RRT $\mathcal{T}_n$. Indeed, conditional on the currently explored cluster containing $k$ vertices, the memoryless property of the exponential distribution implies that the residual passage times along all $k(n-k)$ boundary edges are i.i.d. exponentials. Consequently, the boundary edge through which the next vertex is discovered is uniformly distributed among these edges, and the next vertex attaches to a uniformly chosen previously discovered vertex.
In particular, $H_n$ has the same distribution as the maximal hopcount
\[
 \max_{v\in V(K_n)} \#\{\text{edges in the shortest path from }o\text{ to }v\}.
\]

One may view the vertices as cities, the root $o$ as a central city, and each edge weight as the direct travel time between two cities. The shortest-path tree then records the fastest route from the central city to every other city. Its height records the largest number of legs used by any such fastest route, rather than the largest total travel time. Thus, for every $0<\alpha<1$,
\[
\P\left(\text{the fastest route to every city uses at most }\alpha e\ln n\text{ legs}\right)
=\exp\left(-\omega_\alpha(n)n^{1-\alpha}(\ln n)^{-\frac{3}{2e}}\right). 
\]
where $1 \ll \omega_{\alpha}(n) \ll \ln^{(k)} n$ as $n \to \infty$ for any fixed $k$. Similarly, for every $\beta>1$,
\[
\P\left(\text{the fastest route to some city uses at least }\beta e\ln n\text{ legs}\right)
=n^{-e\beta\ln\beta+o(1)}.
\]
\end{remark}
  
%  \begin{remark}
  
% \heng{A natural  guess of  correct decay rate should be like  
%  \begin{equation}
%    \P \left(  H_{n} \leq \alpha  e \ln n   \right) =  \exp\left(  - [C_{\alpha} +o(1)] \frac{n^{1- \alpha} }{(\ln n)^{3/(2e)} } \right) 
%  \end{equation}
%  But Maybe it is not the case, we should also think about showing that there exists a function $\omega(n) \to \infty$ (but very slow, like the log-height of $n$, defined as $\min \{ k: \ln^{(k)} (n) \le 1 \}$) such that 
%   \begin{equation}
%    \P \left(  H_{n} \leq \alpha  e \ln n   \right) \le   \exp\left(  -  \omega(n)\frac{n^{1- \alpha} }{(\ln n)^{3/(2e)} } \right) 
%  \end{equation}
%  }
% \end{remark}
\medskip
   \subsection{Related work, comments, and questions}  
The random recursive tree (RRT) is a classical model of randomly growing trees
that has been extensively studied since its introduction by Na and
Rapoport \cite{NR70}.  These structures have emerged in a variety of distinct contexts, having been employed to model the spread of epidemics \cite{Moon74}, to reconstruct the stemmatic relationships among preserved copies of ancient manuscripts \cite{NH82}, and to analyze the dynamics of chain letter schemes and pyramid games \cite{GB84}.
Beyond the study of height and depth, the literature has investigated other structural properties of RRTs, including
profiles \cite{FHN06,KMS17,IK18}, asymptotic degree distributions \cite{Janson05}, maximal and
near-maximal degrees \cite{DL95,GS02,ABE18}, subtree size profiles \cite{Fuchs08}, total path length \cite{DF99},
independence numbers \cite{Janson20}, and the root-finding problem \cite{BDL17,ABFKLT24}.
Furthermore, stochastic processes on RRTs have been the subject of intensive
investigation; see e.g. \cite{Baur16,ABDLV22,AL25}.  While the literature cited here is far from exhaustive, it illustrates the breadth of this active and continuously growing field of study.

Various generalizations of RRTs have been  widely studied  in the literature. One notable model is the preferential attachment tree, where  a new incoming vertex samples a parent with probability proportional to the candidate parent's degree.   The binary search tree is another closely related model, which is a sequence of random subtrees of the infinite binary
tree, recursively built by adding new vertices uniformly at random among all possible sites.
Many properties of these
random trees can be analyzed using a unified approach. We refer to
\cite{HJ17} for an excellent survey on this topic.

Connecting recursive trees to Crump-Mode-Jagers (CMJ) branching processes
via continuous-time embedding has been a standard tool
since the   works of Pittel \cite{Pittel94}
and Devroye \cite{Devroye87,Dev98}. 
Consider a Yule process starting with a single ancestor at $t=0$,
where each individual independently gives birth at rate one.
Let $(\tau_{n})_{n \ge 1}$ denote the chronological sequence of birth times,
meaning $\tau_{n} $ is the exact moment the population reaches size $n$.
The sequence of family trees drawn at these discrete moments $\tau_n$
(with chronologically labeled vertices) has the same joint distribution
as the RRT   $\mathcal{T}_{n}$. 
In particular, the height   $H_n$ corresponds to the maximum generation
born by time $\tau_n$. 
Since  $\tau_n-\ln n=O_{\P}(1)$,
analyzing the typical height reduces
to studying the asymptotic first birth time of the $k$-th generation
in a Yule process, which was studied in \cite{Kingman75} and is also subsumed under the general theory of the minimum
of branching random walks \cite{Aidekon13,ABF13}.
 
While this continuous-time embedding is highly efficient for determining
the typical behavior of random trees, it becomes inconvenient
when investigating atypical behavior, such as large deviations. 
The main issue is that the continuous-time model introduces
auxiliary randomness through the exponential birth times.
If we rely on this embedding to study rare  events,
the optimal strategy leading to the rare event would involve controlling this artificial time randomness.
Conditioning on a specific stopping time like
 $\tau_{n}= a \ln n$  changes the distribution of the branching process
and seems to make the analysis   more complicated. 
 Consequently, studying large deviations requires us to  analyze the discrete RRT directly. This loss of a unified framework means we must tackle different random tree structures on a case-by-case basis, which leads to the following question:

   \begin{question}
  Let $\mathcal{T}_n$ be a preferential attachment tree or a binary search tree on $n$ vertices. From   \cite{Pittel94}, it is known that there exists $c_{*}>0$ such that $H(\mathcal{T}_n)/\ln n \to c_{*}$ almost surely.  
  Fix $0< \alpha < 1< \beta $.  What are the decay rates of the probabilities of the rare events  $\{ H_{n} < \alpha c_{*} \ln n \}$ and    $ \{H_{n} > \beta c_{*} \ln n\}$?
  \end{question}

  Although the typical behaviors of quantities associated with RRTs
are well-studied, their large deviation probabilities
remain largely unexplored in the literature.
One closely related result is \cite{BMS09}, where the authors established
the large deviation principle (LDP) for the number
of leaves $L_n$ in an RRT $\mathcal{T}_n$.
Using analytic methods and Dupuis–Ellis-type path arguments, they proved that  $L_n/n$ satisfies an LDP with speed $n$ and a strictly convex rate function $\Lambda^{*}$, which is
given by the Legendre transform of an explicitly defined function $\Lambda$.
Their findings differ significantly from ours.
In our case, the lower large deviation probability
decays doubly exponentially (in $\ln n$), while the
upper large deviation probability decays exponentially (in $\ln n$).
However, this asymmetry between the lower and upper LDPs
is not very surprising: Indeed, based on \cite{ABF13} and \cite{PS22}, one anticipates that although $(H_n - \mathbb{E}[H_n])_{n \ge 1}$ may fail to converge in distribution due to lattice effects, all of its subsequential weak limits correspond to a (randomly shifted) Gumbel distribution. This limiting distribution is characterized by an exponential right tail and a doubly exponential left tail.
Similar phenomena have been observed in other contexts, such as the maximum of the two-dimensional Gaussian free field \cite{ding13} and the maximum of branching random walks with bounded step sizes \cite{CH20}. The underlying intuition is straightforward: achieving an unusually tall tree requires only a single branch to become exceptionally long, which constitutes a local constraint. In contrast, producing an unusually short tree necessitates that all branches remain short, imposing a global constraint and thereby incurring a significantly higher probabilistic cost.

The appearance of the pathological prefactor   $\omega_{\alpha}(n)$  in the lower large deviation rate
is particularly surprising to us.
 As noted previously, this term diverges to infinity, yet its growth is slower than any finite iteration of the logarithm. 
To the best of our knowledge,
such a term appears to be quite novel
within the existing LDP literature.
We do not currently have an explicit
conjecture for its precise growth rate.
However, we suspect that  $\omega_{\alpha}(n)$ is closely related to the following function $\mathtt{h}$ which is defined as follows:
Let $a_{n}=e^{a_{n-1}}$ for $n \ge 2$, with $a_1=1$. Define $\mathtt{h}(n)=\min\{k:  a_{k} \le n \le a_{k+1} \}$. 
It is highly likely that $\omega_{\alpha}(n)$
is comparable to some function of $\mathtt{h}$, like $\omega_{\alpha}(n) \asymp \mathtt{h}(n)^2$ or $\omega_{\alpha}(n) \asymp e^{\mathtt{h}(n)}$. Determining the exact growth rate of  $\omega_{\alpha}(n)$ is left as an interesting open problem
for future investigation.

Finally, large deviation probabilities quantify the cost of an atypical height but do not determine the typical geometry under the corresponding conditioning. The local-versus-global distinction between the two tails suggests that the conditioned trees should have substantially different structures. Natural observables include their level profiles, degree sequences. Understanding these conditioned geometries remains an open interesting direction.

  \begin{question}[Conditional Structure]
    What can we say about the statistics of the tree $\T_{n}$ conditionally on  the unusually short event $\{H_{n} < \alpha e\ln n\}$?  
    For instance, what is the asymptotic behavior of the level set $X_k(n)=\sum_{j=1}^{n} \ind{\mathrm{dist}(v_1,v_j)=k}$, under the conditioned probability $\P(  \cdot \mid H_{n} < \alpha e\ln n  )$?
  \end{question} 
 
  \section{Lower Large Deviations}

  This section is devoted to establishing the following two estimates, from which 
Theorem \ref{thm:lowerLDP}   follows directly.

\begin{proposition}\label{prop:lowerLDP}
 For any fixed integer $k \ge 2$ and any   $0< \alpha < 1$ , there exists a constant $C>0$   such that for all sufficiently large $n$,
  \begin{equation}
   \P \left(  H_{n} \leq \alpha e \ln n \right) \ge    \exp\Big(  - C\, {n^{1- \alpha} }{(\ln n)^{-\frac{3}{2e}} } \ln^{(k)} n \Big) .
  \end{equation}  
\end{proposition}
  
%   The first step is to prove that 
%      \begin{equation}
%     \P \left(  H_{n} \leq \alpha   \ln n \right) =   \exp\left(  - n^{c(\alpha)+o(1)}  \right) 
%  \end{equation} 
%  \begin{TBD}
%   So the second  step is to prove that 
%  \begin{equation}
%     \exp\left(  - C_{-} n^{c(\alpha)} \right) \leq 
% \P \left(  H_{n} \leq \alpha   \ln n \right) \leq    \exp\left(  - C_{+} n^{c(\alpha)}  \right) 
% \end{equation} 
%  The ideal goal  is to prove that 
%  \begin{equation} 
% \P \left(  H_{n} \leq \alpha   \ln n \right) =  \exp\left(  - C_{\alpha} n^{c(\alpha)}  \right) 
% \end{equation}
%  \end{TBD}

\begin{proposition}\label{prop:lowerLDPup}
For any fixed $0< \alpha < 1$, we have   
  \begin{equation}
  \lim_{n \to \infty} \frac{ - \ln  \P \left(  H_{n} \leq \alpha e \ln n \right) }{ n^{1-\alpha} (\ln n)^{-\frac{3}{2e}}  }  = \infty  .
  \end{equation}  
\end{proposition}

\begin{proof}[Proof of Theorem \ref{thm:lowerLDP}]
  Fix $k \in \mathbb{N}$. By Proposition \ref{prop:lowerLDP}, there exists $C>0$ such that 
  $ \omega_{\alpha}(n) \le C \ln^{(k+1)}(n)  $ for all sufficiently large $n$. Thus $ { \omega_{\alpha} (n) }/{ \ln^{(k)} (n)} \to 0$ as $n \to \infty$. Proposition \ref{prop:lowerLDPup} yields $\omega_{\alpha}(n) \to \infty$. This completes the proof.
\end{proof}

\subsection{Proof of Proposition \ref{prop:lowerLDP}} 
For the lower bound, we employ the following  strategy, although a little bit  crude, to ensure that $H_{n} \leq \alpha e \ln n$. 

 Fix a positive integer $m \le n$ to be determined later.
  For each $1 \le j \leq m$, let $  \T^{(j)}_{m,n} $ denote the connected component of $\T_n$ that contains the vertex $v_j$, after removing all edges between vertices in $\{ v_i: i\in [m]\}$. See Figure \ref{fig:Tmn} for an illustration.
  Clearly, each $\T_{m,n}^{(j)}$ is a tree, and we take $v_j$ as its root. 
   According to the uniform attachment rule,
  $(\T_{m,n}^{(j)} , 1 \leq j \leq m)$ are independent of the tree $\T_{m}$. Moreover,   conditionally on their sizes $(|\T_{m,n}^{(j)}|)_{j=1}^m$, the trees  $(\T_{m,n}^{(j)} , 1 \leq j \leq m)$ are independent random recursive trees on $|\T_{m,n}^{(j)} |$ vertices.
  % (after modifying the labels of each vertices and keep the order accordingly). 

  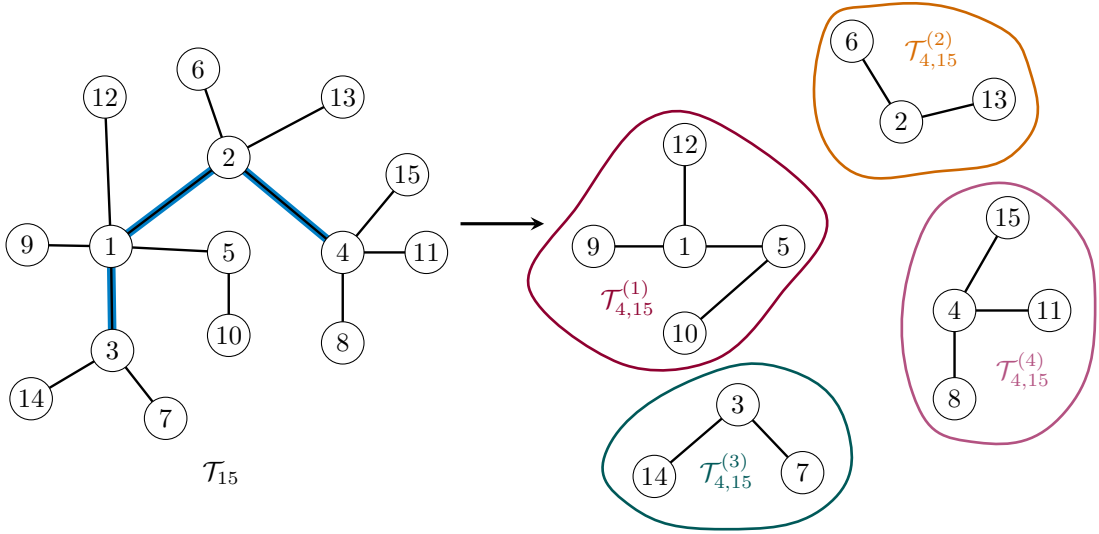
\begin{figure}[t]
  \centering
  \begin{tikzpicture}[
    >=stealth,
    x=1cm,
    y=1cm,
    every node/.style={font=\small},
    v/.style={circle, draw, fill=white, minimum size=5.6mm, inner sep=0pt},
    cut/.style={draw=cyan!60!blue, line width=3.2pt, line cap=round},
    edge/.style={draw=black, line width=0.9pt},
    forest/.style={draw, line width=1pt, line cap=round, line join=round},
    lab/.style={font=\small},
    tlab/.style={font=\small, inner sep=1pt}
  ]
    % Left: the tree \mathcal{T}_{15}
    \coordinate (L1) at (0,0);
    \coordinate (L2) at (1.55,1.18);
    \coordinate (L3) at (0.02,-1.38);
    \coordinate (L4) at (3.05,-0.08);
    \coordinate (L5) at (1.55,-0.08);
    \coordinate (L6) at (1.15,2.35);
    \coordinate (L7) at (0.72,-2.28);
    \coordinate (L8) at (3.05,-1.28);
    \coordinate (L9) at (-1.1,0.02);
    \coordinate (L10) at (1.55,-1.18);
    \coordinate (L11) at (4.15,-0.08);
    \coordinate (L12) at (-0.08,1.98);
    \coordinate (L13) at (3.05,1.98);
    \coordinate (L14) at (-1.05,-2.02);
    \coordinate (L15) at (3.9,0.95);

    \draw[cut] (L1)--(L2);
    \draw[cut] (L1)--(L3);
    \draw[cut] (L2)--(L4);

    \draw[edge] (L1)--(L9);
    \draw[edge] (L1)--(L12);
    \draw[edge] (L1)--(L2);
    \draw[edge] (L1)--(L5);
    \draw[edge] (L1)--(L3);
    \draw[edge] (L3)--(L14);
    \draw[edge] (L3)--(L7);
    \draw[edge] (L2)--(L6);
    \draw[edge] (L2)--(L13);
    \draw[edge] (L2)--(L4);
    \draw[edge] (L4)--(L15);
    \draw[edge] (L4)--(L11);
    \draw[edge] (L4)--(L8);
    \draw[edge] (L5)--(L10);

    \foreach \n/\p in {
      1/L1,2/L2,3/L3,4/L4,5/L5,6/L6,7/L7,8/L8,9/L9,10/L10,11/L11,12/L12,13/L13,14/L14,15/L15}
      \node[v] at (\p) {$\n$};

    \node[lab] at (1.45,-3) {$\mathcal{T}_{15}$};

    % Arrow
    \draw[->, line width=1.2pt] (4.6,0.3) -- (5.7,0.3);

    % Right: components T^{(j)}_{4,15}
    % component 1
    \coordinate (A1) at (7.55,0.0);
    \coordinate (A5) at (8.85,0.0);
    \coordinate (A9) at (6.35,0.0);
    \coordinate (A10) at (7.55,-1.15);
    \coordinate (A12) at (7.55,1.35);
    \draw[edge] (A1)--(A5);
    \draw[edge] (A1)--(A9);
    \draw[edge] (A1)--(A12);
    \draw[edge] (A5)--(A10);
    \foreach \n/\p in {1/A1,5/A5,9/A9,10/A10,12/A12}
      \node[v] at (\p) {$\n$};
    \draw[purple!75!black, forest]
      plot[smooth cycle, tension=0.6] coordinates {
        (5.5,-0.55) (6.15,0.55) (7.0,1.65) (7.95,1.65)
        (9.4,0.45) (8.8,-0.65) (7.95,-1.5) (6.7,-1.55)
      };
    \node[tlab, text=purple!75!black] at (6.8,-0.7) {$\mathcal{T}^{(1)}_{4,15}$};

    % component 2
    \coordinate (B2) at (10.4,1.65);
    \coordinate (B6) at (9.75,2.72);
    \coordinate (B13) at (11.62,1.95);
    \draw[edge] (B2)--(B6);
    \draw[edge] (B2)--(B13);
    \foreach \n/\p in {2/B2,6/B6,13/B13}
      \node[v] at (\p) {$\n$};
    \draw[orange!80!black, forest]
      plot[smooth cycle, tension=0.85] coordinates {
        (9.55,1.15) (9.35,2.75) (10.15,3.2) (11.5,3.0) (12.1,2.3) (12.0,1.35) (10.95,1.0)
      };
    \node[tlab, text=orange!85!black] at (10.8,2.6) {$\mathcal{T}^{(2)}_{4,15}$};

    % component 3
    \coordinate (C3) at (8.25,-2.1);
    \coordinate (C14) at (7.15,-3.05);
    \coordinate (C7) at (9.1,-3.0);
    \draw[edge] (C3)--(C14);
    \draw[edge] (C3)--(C7);
    \foreach \n/\p in {3/C3,14/C14,7/C7}
      \node[v] at (\p) {$\n$};
    \draw[teal!70!black, forest]
      plot[smooth cycle, tension=0.8] coordinates {
        (6.55,-2.65) (6.85,-3.45) (8.05,-3.75) (9.4,-3.45) (9.7,-2.45) (8.95,-1.65) (7.7,-1.75)
      };
    \node[tlab, text=teal!70!black] at (8.1,-3) {$\mathcal{T}^{(3)}_{4,15}$};

    % component 4
    \coordinate (D4) at (11.1,-0.85);
    \coordinate (D15) at (11.8,0.38);
    \coordinate (D11) at (12.35,-0.85);
    \coordinate (D8) at (11.1,-2.02);
    \draw[edge] (D4)--(D15);
    \draw[edge] (D4)--(D11);
    \draw[edge] (D4)--(D8);
    \foreach \n/\p in {4/D4,15/D15,11/D11,8/D8}
      \node[v] at (\p) {$\n$};
    \draw[magenta!70!black, forest]
      plot[smooth cycle, tension=0.82] coordinates {
        (10.55,-1.95) (10.55,-0.15) (11.4,0.8) (12.55,0.4) (12.95,-0.8) (12.55,-2.2) (11.4,-2.6)
      };
    \node[tlab, text=magenta!70!black] at (12,-1.65) {$\mathcal{T}^{(4)}_{4,15}$};
  \end{tikzpicture}
  \caption{Example of the forest $(\mathcal{T}^{(j)}_{4,15})_{1\le j\le 4}$ obtained from $\mathcal{T}_{15}$ by removing the edges between vertices in $V_{[4]}=\{1,2,3,4\}$. The highlighted blue edges on the left are precisely the removed edges.}
  \label{fig:Tmn}
\end{figure}

  \begin{lemma} \label{lem:subtree-size-1}
 There exists an absolute constant $M$  such that  
\begin{equation}
\P\left(  \max_{1 \leq j \leq m}| \T_{m,n}^{(j)} | \leq  2  \frac{n}{m}    \right) \geq e^{- m}  \,  \text{ for all }  n/M \geq m \geq M .\label{eq:lower-strategy-2} 
\end{equation}
  \end{lemma}
 
  \begin{lemma}
    \label{lem:rrt-heightk}
    Let $A_{k}(n)$ denote the number of increasing trees on $n$ vertices with height less than or equal to $k$.  There exists a constant $C_{k}$ satisfying
\begin{equation}\label{eq-counting-Akn}
  A_{k}(n) \ge \exp \left(   n \ln n-n \ln^{(k)}(n) -C_{k} n \right) \, \text{ for all large } n.
\end{equation} 
  \end{lemma}
  
\begin{proof}[Proof of Proposition \ref{prop:lowerLDP} admitting Lemmas \ref{lem:subtree-size-1} and \ref{lem:rrt-heightk}] 

Fix a large constant $k$ independent of $n$. Our strategy is first to 
 force the height of  $\mathcal{T}_m$ to be less than $k$, i.e.,  $H(\mathcal{T}_{m}) \leq k$; and then to require that $  H( |\T_{m,n}^{(j)} | ) \leq \alpha e \ln n-k$ for all $1 \le j \le m$. It's clear then we have $H_{n} \leq \alpha e \ln n$.  Therefore we obtain
\begin{align}
  \P \left(  H_{n} \leq \alpha e \ln n \right) & \geq 
  \P \left( \mathrm{dist}(v_1,v_j) \leq k ,  H( |\T_{m,n}^{(j)} | )\leq \alpha e \ln n -k,  \forall 1 \leq j\leq m \right)  \\
  & \geq  \mathsf{F}(m,k) \, \E \bigg[  \prod_{j=1}^{m}  \mathsf{F}( |\T_{m,n}^{(j)} |, \alpha e \ln n -k)  \bigg], 
\end{align}
where $\mathsf{F}(n,x)\coloneq \P(H_{n} \leq x) \le 1$.  Notice that $\mathsf{F}(n,x)  $ is decreasing in $n$ for fixed $x$.
On the event $\{  \max_{1 \leq j \leq m} |\T_{m,n}^{(j)} | \leq 2 \frac{n}{m} \}$, we have 
\begin{equation}
  \P \left(  H_{n} \leq \alpha e\ln n \right)   \geq \mathsf{F}(m,k)\P\left(  \max_{1 \leq j \leq m}| \T_{m,n}^{(j)} | \leq 2 \frac{n}{m} \right)  \mathsf{F}\left( 2 \frac{n}{m}, \alpha e\ln n -k\right)^{m} .  \label{eq:lower-strategy}
\end{equation}

In the following let $K>0$ be a large constant to be chosen later and  set 
\begin{equation}
    m=m_{K,n} \coloneq 2K n^{1- \alpha } (\ln  n)^{-3 /(2e)}  \ \text{ so that } \  \frac{2n}{m} = K^{-1}  n^{\alpha} \, (\ln n)^{3 /(2e) } .
\end{equation}   According to  \eqref{eq-ABF-EH-bound}, we obtain that provided $n$ is sufficiently large,
\begin{align}
  \E \big[ H_{ {2n}/{m_{K,n}} } \big]
  & = e \ln \big(  K^{-1} n^{\alpha} (\ln  n)^{3  /2e} \big) - \frac{3}{2}  \ln \big(   \ln (n^{\alpha+o_{n}(1)}) \big)  \\
  & = \alpha e \ln n -  e \ln K- \frac{3}{2} \ln \alpha + o_{n}(1)  
\end{align}
The tightness of $(H_{n}- \E[H_{n}])_{n \ge 1}$ yields that  provided $e \ln K + \frac{3}{2} \ln \alpha -k$ is sufficiently large, then  we have 
\begin{align}
& 1- \mathsf{F} \Big( \frac{2n}{m_{K,n}}, \alpha e \ln n - k \Big)   =   \P \big( H_{ 2n/m_{K,n}}  > \alpha e \ln n - k\big)   \\
  & \qquad  \leq   \P \Big( H_{ 2n/m_{K,n}}  > \E[ H_{ 2n/m_{K,n} }] + e \ln K + \frac{3}{2} \ln \alpha -k \Big)  \leq 1-1/e . 
  \label{eq:lower-strategy-1} 
\end{align}  
 for all large $n$. 
Plugging the inequalities \eqref {eq:lower-strategy-1}  as well as   \eqref {eq:lower-strategy-2}   back into \eqref{eq:lower-strategy}  yields   
\begin{align}
  \P \left(  H_{n} \leq \alpha e \ln n \right) 
  & \geq \mathsf{F}( m_{K,n},k)      e^{-2 m_{K,n}} . 
\end{align}

Note that we have $\mathsf{F}(m,k)=A_{k}(m)/(m-1)! \ge \exp( - m \ln ^{(k)} m -\Theta(m) )$ by using Lemma \ref{lem:rrt-heightk}. Finally  we  conclude  
\begin{align}
  \P \left(  H_{n} \leq \alpha e \ln n \right) 
  & \geq \mathsf{F}( m_{K,n},k)      e^{-2 m_{K,n}} \ge   e^{- m_{K,n} \ln^{(k)}  m_{K,n}} \ e^{- \Theta(m_{K,n})}  \\
  &=  \exp\left(   - \Theta ( \ln^{(k)} n ) n^{1-\alpha} (\ln n)^{-3/2e}\right).
\end{align}
This establishes the lower bound in Proposition \ref{prop:lowerLDP}.
\end{proof}

\medskip 

\begin{proof}[Proof of Lemma \ref{lem:subtree-size-1}] 
% It remains to show that   $\inf_{n \geq m \geq 2}\P\left(  \max_{1 \leq j \leq m}| \T_{m,n}^{(j)} | \leq  2 \frac{n}{m}    \right) \geq e^{- \Theta(m)}$. 
Consider the collection $ (|\T_{m,n}^{(j)} |: 1 \leq j \leq m) $ as a process indexed by $n \geq m$.  For the initial time $n=m$, we have
   $ | \T_{m,m}^{(j)} | =1 $  for all $1 \leq j \leq m$. Moreover given the sizes $  |\T_{m,n}^{(j)} | : 1 \leq j \leq m$, the evolution of the process is described by  
\begin{equation}
   |\T_{m,n+1}^{(j)} | = |\T_{m,n}^{(j)} | + \ind{ j= \mathcal{V}} \text{ for all }  1 \leq j \leq m
\end{equation} 
where the random index $\mathcal{V}$ is chosen with conditional probabilities $\P( \mathcal{V}= v_k \mid |\T_{m,n}^{(j)} | , j \leq m )= \frac{|\T_{m,n}^{(k)} |}{n}$ for all $1 \leq k \leq m$. 
In other words, $ (|\T_{m,n}^{(j)} |: 1 \leq j \leq m) $  has the same distribution as a Pólya urn model with $n$ balls, $m$ types, and an initial configuration of $(1,1,\dots,1)$. Then notice that the Pólya urn model is exchangeable:
for each possible configuration $(n_j)_{j=1}^{m} $, we have 
\begin{align}
  \P( |\T_{m,n}^{(j)} | = n_{j} : 1 \leq j \leq m  ) &=  \binom{n - m}{n_{1}-1, \cdots, n_{m}-1  } \frac{ (n_{1} -1 )! \cdots (n_{m} -1 )!}{m(m+1) \cdots (n-1)} \\
  &= \frac{(n-m)! (m-1)!}{(n-1)!} =\frac{1}{\binom{n-1}{m-1}}. 
\end{align} 
In other words $ (|\T_{m,n}^{(j)} |: 1 \leq j \leq m) $ is uniformly distributed over all compositions of $n$ of size $m$
\begin{equation}\label{eq:subtree-size-distri}
\mathfrak{C}_{n,m}\coloneq \bigg\{ (n_j)_{j=1}^{m}  : n_{j} \in \mathbb{N},   \sum_{j=1}^{m} n_j = n   \bigg\}. 
\end{equation} 

For the remainder of the proof, set $q\coloneq n/m$ and let 
 $ \mathfrak{G}_{n,m}$ denote the subset of $ \mathfrak{C}_{n,m}$ such that   $n_{j} \in [q/4,7q/4] $ for every $1 \leq j \leq m-1$ and $ \sum_{j=1}^{m-1} n_{j} \geq  n- 2q $. Thus every composition in ${\mathfrak{G}}_{n,m}$ has all its coordinates at most $2q$.  We claim that, for all sufficiently large $m$ and $q$,
\begin{equation}\label{eq-card-G-nm}
 |{\mathfrak{G}}_{n,m}|\geq q^{m-1} = (\frac{n}{m})^{m-1}
\end{equation} 
Consequently, since $mq=n$ we have 
\begin{align}
 \P\left(\max_{1\leq j\leq m}|\T_{m,n}^{(j)}|\leq 2\frac{n}{m}\right)
 &\geq\frac{|{\mathfrak{G}}_{n,m}|}{|\mathfrak{C}_{n,m}|} 
 \geq\frac{ (n/m)^{m-1}}{\binom{n-1}{m-1}}\\
 &
 \geq\frac{(m-1)!}{m^{m-1}}\geq e^{-m},
\end{align}
where the last inequality follows from Stirling's lower bound.

It remains to prove \eqref{eq-card-G-nm}.  Put $r\coloneq 3q/4-1$, let $Q\sim\operatorname{Unif}[q-r,q+r]$ and $U\sim\operatorname{Unif}[0,1]$ be independent, and define
\[
 N\coloneq\lfloor Q\rfloor+\ind{U\leq Q-\lfloor Q\rfloor} \in  [q/4,7q/4]\cap\mathbb{N}
\]
This definition guarantees $\E[N\mid Q]=Q$, and hence $\E N=q$.  Moreover,  for every integer $k$,   we have 
\[
\begin{aligned}
\P(N=k)
&\le \frac1{2r} \Bigl[ 
\int_{k-1}^{k}(x-k+1)\, \dif x 
+
\int_k^{k+1}(k+1-x)\, \dif x 
 \Bigr] = \frac{1}{2r} .
\end{aligned}
\] 
Let $N_1,\ldots,N_{m-1}$ be independent copies of $N$; and let $N_m:= mq- \sum_{j=1}^{m-1} N_{j} $. 
By construction,  
\[
(N_j)_{j=1}^{m} \in {\mathfrak{G}}_{n,m} \quad \text{ if and only if } \quad  mq- 2q \le \sum_{j=1}^{m-1}  N_{j}  \le  mq-1 . 
\]
Putting  $X_j=(N_j-q)/q$, we can rewrite the condition above as $\sum_{j=1}^{m-1} X_j \in [ -1,1-q^{-1}]$. This gives  
\begin{align}
 P_{q}(m) &\coloneq \P \biggl( \sum_{j=1}^{m-1} X_j \in [ -1,1-q^{-1}]   \biggr) = \P( (N_j)_{j=1}^{m} \in {\mathfrak{G}}_{n,m}  )   \\
  & = \sum_{ (n_{j})_{j=1}^{m} \in \mathfrak{G} _{n,m} }  \prod_{j=1}^{m-1}\P(  N_j = n_j  ) \le \frac{1}{(2r)^{m-1}} |\mathfrak{G} _{n,m} | . \label{eq:bound-p-q-m}
\end{align} 
Write $ \sigma_q^2\coloneq\Var(X)
 $ and $
 \rho_q\coloneq\E|X|^3$.
The Berry--Esseen inequality (see \cite{Shevtsova11}) yields that there is  a  universal constant $C_0  <1/2$ such that 
\begin{align}    P_{q}(m)  \ge  \int_{- \frac{1}{ \sigma_q \sqrt{(m-1) }}}^{\frac{1-1/q}{\sigma_q \sqrt{(m-1)  }}} e^{-\frac{x^2}{2}} \frac{\dif x}{\sqrt{2 \pi}}  - \frac{2 C_0}{\sqrt{m-1}} \frac{\rho_q}{\sigma_q^3}. 
\end{align}
Since $|N-Q|\leq1$ and $(Q-q)/q$ is uniform on $[-r/q,r/q]$, we have, as $q\to\infty$, 
$ \sigma_q^2=\frac{3}{16}+O(q^{-1}) $,  and  $
 \frac{\rho_q}{\sigma_q^3}=\frac{3\sqrt3}{4}+O(q^{-1}). $
Thus, denoting by $\Phi$ the distribution function of a standard Gaussian random variable, the Berry--Esseen inequality gives
\begin{equation}
 P_{q}(m) \ge 
 \frac{1}{\sqrt{m-1}}\left(
 \frac{8}{\sqrt{6\pi}}
 -2C_0\frac{3\sqrt3}{4}
 +o_{q \wedge m}(1)\right)
 \geq\frac{1}{4\sqrt{m}}
\end{equation}
for sufficiently large $m$ and $q$.  Here the last inequality follows from $C_0\leq 1/2$ and
$
 \frac{8}{\sqrt{6\pi}}- \frac{3\sqrt3}{4}>\frac14$.
Combining this estimate with \eqref{eq:bound-p-q-m}  yields
\[
 |\mathfrak{G}_{n,m}|
 \geq\frac{(2r)^{m-1}}{4\sqrt m}  \geq q^{m-1},
\]
provided $q$, $m$ are large enough.
This proves \eqref{eq-card-G-nm} and completes the proof. 
\end{proof} 
\medskip 

\begin{figure}[t]
  \centering
  \begin{tikzpicture}[
    >=stealth,
    x=1cm,
    y=1cm,
    every node/.style={font=\small},
    big/.style={circle, draw, fill=gray!12, minimum size=1.18cm, inner sep=0pt},
    mid/.style={circle, draw, fill=gray!12, minimum size=0.95cm, inner sep=0pt},
    small/.style={circle, draw, fill=gray!12, minimum size=0.42cm, inner sep=0pt},
    rem/.style={circle, draw=violet, fill=violet!10, minimum size=0.95cm, inner sep=0pt},
    remsmall/.style={circle, draw=violet, fill=violet!10, minimum size=0.42cm, inner sep=0pt},
    brace/.style={decorate, decoration={brace, amplitude=5pt}},
    vtext/.style={text=violet}
  ]
    \node[big] (A) at (0,0) {$A$};
    \node at (-0.05,-1.05) {size};
    \node at (-0.05,-1.48) {$N+1$};

    \node[mid] (M1) at (4.2,2.15) {};
    \node[mid] (M2) at (4.2,0.85) {};
    \node at (4.2,-0.45) {$\vdots$};
    \node[mid] (M3) at (4.2,-1.85) {};
    \node[rem] (MR) at (4.0,-3.35) {};

    \draw[->] (A.25) to[out=25,in=180] (M1.180);
    \draw[->] (A.8) to[out=8,in=180] (M2.180);
    \draw[->] (A.-12) to[out=-12,in=180] (M3.180);
    \draw[->,violet] (A.-32) to[out=-32,in=180] (MR.180);

    \node at (4.18,3.03) {size $s(N)+1$};
    \node[vtext] at (4.18,-4.18) {size $r(N)$};
    \draw[dashed, decorate, decoration={brace, amplitude=8pt}] 
    (5.05,2.65) -- (5.05,-4)
    node[midway, right=0.2cm, align=center] {$\ d(N)$ many};
    \node at (5.92,1.1) {$\vdots$};
     \node at (5.92,0.4) {$\vdots$};
    \node at (5.92,-1.05) {$\vdots$};
    \node at (5.92,-2.05) {$\vdots$};
    \node at (5.92,-3.05) {$\vdots$};

    \node[small] (S1) at (8.45,3.55) {};
    \node[small] (S2) at (8.45,2.78) {};
    \node at (8.45,2.05) {$\vdots$};
    \node[small] (S3) at (8.45,1.25) {};
    \node[remsmall] (SR) at (8.45,0.48) {};

    \draw[->] (M1.20) to[out=20,in=180] (S1.180);
    \draw[->] (M1.8) to[out=8,in=180] (S2.180);
    \draw[->] (M1.-8) to[out=-8,in=180] (S3.180);
    \draw[->,violet] (M1.-18) to[out=-18,in=180] (SR.180);

    \node at (8.52,4.25) {size $s(s(N))+1$};
   \draw[decorate, decoration={brace, amplitude=10pt}, dashed] 
    (8.9,3.9) -- (8.9,0.1)
    node[midway, right=0.16cm, align=center] { $\quad  d(s(N))$ many};
    \node[vtext] at (9.12,-0.35) {size $r(s(N))$};

    \node at (9.95,3.8) {$\vdots$};
    \node at (9.95,2.7) {$\vdots$};
    \node at (9.95,1.5) {$\vdots$};
    \node at (9.95,0.7) {$\vdots$};
  \end{tikzpicture}
  \caption{A partition scheme: from a set $A$ of size $N+1$, one obtains $d(N)-1$ blocks of size $s(N)+1$ and one remainder block of size $r(N)$; iterating the same rule on a size-$s(N)+1$ block yields $d(s(N))-1$ blocks of size $s(s(N))+1$ and one remainder block of size $r(s(N))$.}
  \label{fig:hk-partition}
\end{figure}
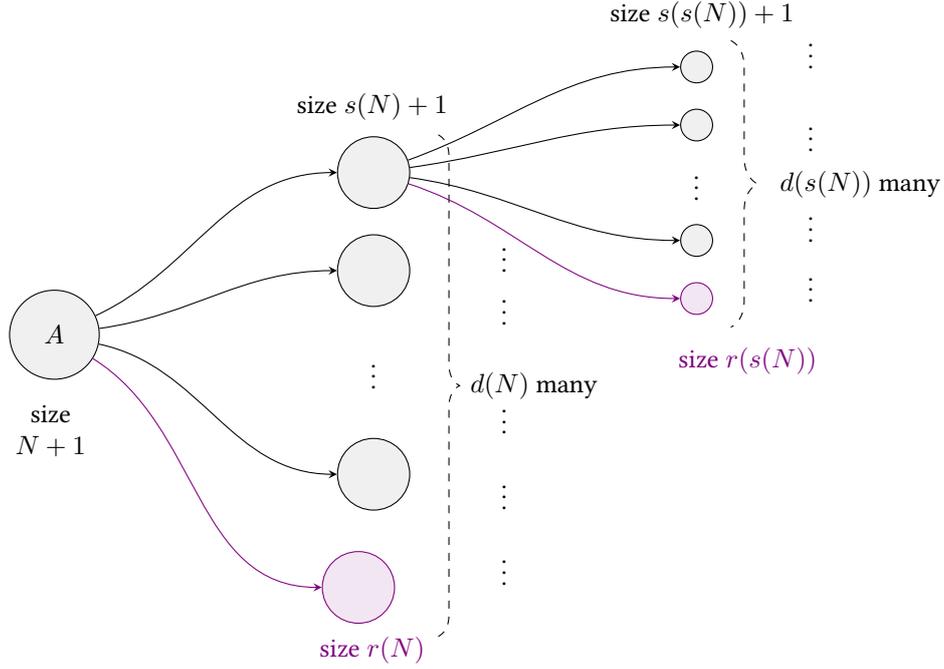

\begin{proof}[Proof of Lemma \ref{lem:rrt-heightk}]
We describe a procedure for constructing a family of increasing trees with  height  $k$. This  relies on a specific partitioning scheme.  
First, for any non-empty subset of vertices  $A \subset V_{[n+1]}\coloneq\{v_{i}\}_{i=1}^{n+1}$,  we identify a unique leader $v^{*}(A)$  with the minimum index. The set of followers is then defined  as 
\begin{equation}
  \hat{A} \coloneq A \setminus \{ v^{*}(A) \} .
\end{equation}
Let $N \coloneq |\hat{A}|$.  Let   $s(N) \coloneq \lfloor \ln N \rfloor$ and 
 $d(N)\coloneq \lfloor \frac{N}{ s(N)+1 } \rfloor$.  Define  $r(N)\coloneq N- d(N) [s(N)+1] \in [0,1+s(N)]$.   
Let  $\mathcal{P}(\hat{A})$ denote  the collection of all partitions of  $\hat{A}$ into $d(|\hat{A}|)$ subsets, where $d(|{\hat{A}}|)-1$ of the subsets have size $s(|\hat{A}|)+1$, and one  has size $1+s(|\hat{A}|)+r(|\hat{A}|)$. See Figure \ref{fig:hk-partition} for an illustration.
  The cardinality of this collection is given by
\begin{equation}
\mathcal{P}_{N} \coloneq  | \mathcal{P}(\hat{A}) |= \left\{
 \begin{aligned}
&\frac{1}{ [d(N)-1] !} \frac{N!}{ [ (1+s(N) )! ]^{d(N)-1} [1+s(N)+r(N)]! }  \  &\text{ if } r(N) \ge 1  \,; \\
&\frac{1}{d(N) !} \frac{N!}{ [ (1+s(N))! ]^{d(N)}   }  \   &\text{ if } r(N) = 0 \,.
 \end{aligned} \right.
\end{equation}
Stirling's approximation yields the asymptotic behavior of $\ln \mathcal{P}_{N}$: 
\begin{equation}\label{eq:asy-P-N}
  \ln    \mathcal{P}_{N}   =  N \left[  \ln N -  \ln^{(2)} N - 1  +o_{N} (1)  \right] \text{ as } N \to \infty .
\end{equation}
 
% \begin{figure}[ht]
%   \includegraphics[width=0.8\textwidth]{hk.jpg}
% \end{figure}

Using this scheme, we now construct the increasing tree on $V_{[n+1]}$ with height $k$. Assume that $k \ge 3$ and $n$ is large so that $\ln^{(k)}(n) \ge 10^{10}$.  We proceed in two stages: 
\begin{enumerate}
  \item A tree labeled by subsets: We first build a tree  $T_{\mathrm{set}}$ inductively where each node is a subset of $V_{[n+1]}$ as follows: The root is $V_{[n+1]}$ itself. Then, for each existing node $A$ at a level $i \in \{0,\cdots,k-2\}$,  we generate its children by selecting a partition uniformly at random from $\mathcal{P}(\hat{A})$. Each set in this partition becomes a child node of $A$.
  \item Mapping to an increasing tree $\mathcal{T}$: We convert the subset tree $T_{\mathrm{set}}$ into the final increasing tree $\mathcal{T}$. For each  node $A$ in $T_{\mathrm{set}}$, relabel it with its leader $v^{*}(A)$. Additionally, if this node $A$ is a leaf in $T_{\mathrm{set}}$,  we add every $v_{j}$ in $\hat{A}$ as a child of $v^{*}(A)$ in $\mathcal{T}$. 
\end{enumerate}
See Figure \ref{fig:hk-construction} for an illustration.
By the construction above, it is clear that $\mathcal{T}$ is an increasing tree on $V_{[n+1]}$ with height $k$ and the mapping $T_{\mathrm{set}} \to \mathcal{T}$ is an injection. 
Now it suffices to estimate the total number of $T_{\mathrm{set}}$.

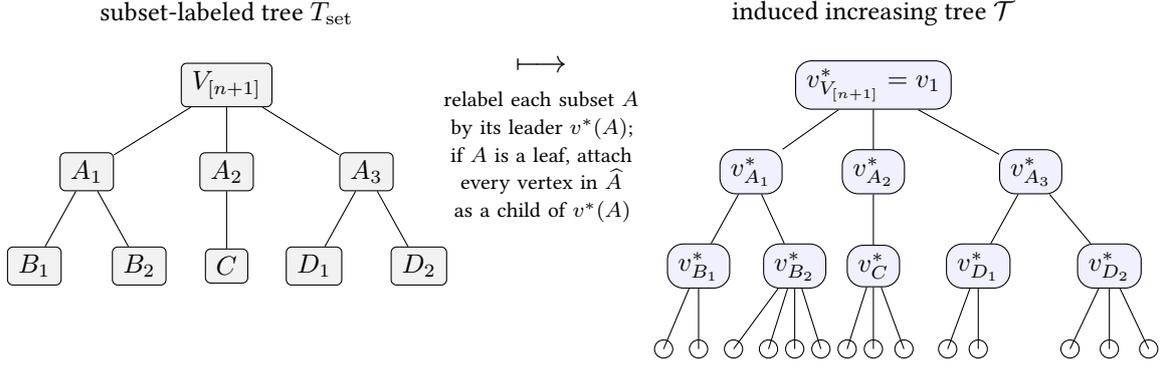
\begin{figure}[t]
  \centering
  \begin{tikzpicture}[
    x=1.15cm,
    y=1cm,
    every node/.style={font=\small},
    subset/.style={draw, rounded corners=2pt, fill=gray!10, inner xsep=4pt, inner ysep=3pt, align=center},
    leader/.style={draw, rounded corners=6pt, fill=blue!6, inner xsep=4pt, inner ysep=3pt, align=center},
    follower/.style={circle, draw, inner sep=0pt, minimum size=2.4mm}
  ]
    \node[align=center] at (1.8,5.2) {subset-labeled tree $T_{\mathrm{set}}$};
    \node[align=center] at (9.2,5.2) {induced increasing tree $\mathcal{T}$};

    \node[subset] (rootl) at (1.8,4.25) {$V_{[n+1]}$};
    \node[subset] (aone) at (0.2,3.1) {$A_1$};
    \node[subset] (atwo) at (1.8,3.1) {$A_2$};
    \node[subset] (athree) at (3.4,3.1) {$A_3$};
    \node[subset] (bone) at (-0.4,1.85) {$B_1$};
    \node[subset] (btwo) at (0.8,1.85) {$B_2$};
    \node[subset] (cone) at (1.8,1.85) {$C$};
    \node[subset] (done) at (2.8,1.85) {$D_1$};
    \node[subset] (dtwo) at (4.0,1.85) {$D_2$};

    \draw (rootl) -- (aone);
    \draw (rootl) -- (atwo);
    \draw (rootl) -- (athree);
    \draw (aone) -- (bone);
    \draw (aone) -- (btwo);
    \draw (atwo) -- (cone);
    \draw (athree) -- (done);
    \draw (athree) -- (dtwo);

    \node[font=\large] at (5.4,4.5) {$\longmapsto$};
    \node[align=center, font=\scriptsize, text width=3.2cm] at (5.4,3.3)
    {relabel each subset $A$ by its leader $v^{*}(A)$;\\if $A$ is a leaf, attach every vertex in $\hat{A}$ as a child of $v^{*}(A)$};

    \node[leader] (rootr) at (9.2,4.25) {$v^{*}_{V_{[n+1]}}=v_1$};
    \node[leader] (rone) at (7.8,3.1) {$v^{*}_{A_1}$};
    \node[leader] (rtwo) at (9.2,3.1) {$v^{*}_{A_2}$};
    \node[leader] (rthree) at (11.0,3.1) {$v^{*}_{A_3}$};
    \node[leader] (sone) at (7.2,1.85) {$v^{*}_{B_1}$};
    \node[leader] (stwo) at (8.3,1.85) {$v^{*}_{B_2}$};
    \node[leader] (sthree) at (9.2,1.85) {$v^{*}_{C}$};
    \node[leader] (sfour) at (10.4,1.85) {$v^{*}_{D_1}$};
    \node[leader] (sfive) at (11.9,1.85) {$v^{*}_{D_2}$};

    \draw (rootr) -- (rone);
    \draw (rootr) -- (rtwo);
    \draw (rootr) -- (rthree);
    \draw (rone) -- (sone);
    \draw (rtwo) -- (sthree);
    \draw (rone) -- (stwo);
    \draw (rthree) -- (sfour);
    \draw (rthree) -- (sfive);

    \foreach \x in {6.8,7.2} {
      \node[follower] at (\x,0.75) {};
      \draw (sone) -- (\x,0.75);
    }
    \foreach \x in {7.6,8.0,8.3,8.6} {
      \node[follower] at (\x,0.75) {};
      \draw (stwo) -- (\x,0.75);
    }
    \foreach \x in {8.9,9.2,9.55} {
      \node[follower] at (\x,0.75) {};
      \draw (sthree) -- (\x,0.75);
    }
    \foreach \x in {10.05,10.4} {
      \node[follower] at (\x,0.75) {};
      \draw (sfour) -- (\x,0.75);
    }
    \foreach \x in {11.45,11.9,12.35} {
      \node[follower] at (\x,0.75) {};
      \draw (sfive) -- (\x,0.75);
    }
  \end{tikzpicture}
    \caption{From subset labeled tree to increasing tree} \label{fig:hk-construction}
\end{figure}

We proceed level by level. For the root $V_{[n+1]}$, we have $ \mathcal{P}_{n} $
many ways to choose a partition in $\mathcal{P}(\hat{V}_{[n]})$. 
Moreover, for each $1 \leq i \leq k-1$, every node at the $i$-th level of $T_{\mathrm{set}}$, as a subset of $V_{[n]}$, contains at least $1+s^{(i)}(n)$ elements.  Here $s^{(i)}$ is the $i$-fold iteration of the $N \mapsto s(N)$. 
Let us consider those nodes that have size exactly $1+s^{(i)}(n)$. By induction, there are at least        
\begin{equation}
L_{i}(n)\coloneq \prod_{j=0}^{i-1} [ d(s^{(j)}(n)) -1 ] = \frac{n}{\ln^{(i)} (n) } \bigg[ 1+ O\Big(\frac{1}{ \ln^{(i)} (n)}\Big) \bigg]
\end{equation} 
many. Thus the total number of all possible configurations $T_{\mathrm{set}}$ is bounded from below by 
\begin{equation}
  \mathcal{P}_{n} \times \prod_{i=1}^{k-2} \ ( \mathcal{P}_{ s^{(i)}(n) } )  ^{ L_{i}(n) } .
\end{equation}
It follows from \eqref{eq:asy-P-N} that 
\begin{align}
 & \ln\big(   [ \mathcal{P}_{ s^{(i)}(n) } ]^{L_{i}(n)}  \big)
 =  L_{i}(n) \ln  \mathcal{P}_{s^{(i)}(n) } \\
 &=   \frac{n}{\ln^{(i)} (n) } \left[ 1+ O(\frac{1}{ \ln^{(i)} (n)}) \right] \ln^{(i)} (n) \left[ \ln^{(i+1)}(n) -  \ln^{(i+2)}(n)+O(1)\right] \\
 &= n \left[   \ln^{(i+1)}(n) -  \ln^{(i+2)}(n)+O(1)  \right].
\end{align}
We obtain that the number of all possible configurations for $T_{\mathrm{set}}$ is bounded from below by
\begin{equation}
  \exp \left(  n \ln n -  n \ln^{(k)} n + O_{k}(n) \right) .
\end{equation}
Therefore this is also a lower bound for $A_{k}(n+1)$. 
 This establishes \eqref{eq-counting-Akn}.
\end{proof}

\subsection{Proof of Proposition \ref{prop:lowerLDPup}} We first state two key lemmas that are needed in the proof. For each $k \ge 1$, let us denote 
\begin{equation}
  X_{k}(n) \coloneq \big| \big\{ 1  \le j \le n:  \mathrm{dist}(v_{1}, v_{j}) = k \big\} \big|  .
\end{equation}

\begin{lemma}\label{lem:kth-level-ldp}
  Fix any $k \in \mathbb{N}$ and fix any $\rho <1$. Then for sufficiently large $n$, we have 
  \begin{equation}
    \P( X_{k}(n)  \ge t ) \le  \exp \big(-\rho\, t \ln^{(k)} t \, \big) \ \text { for all } \  \frac{n}{\ln^{(k)} n } \le  t \le n .
  \end{equation} 
\end{lemma}

Lemma \ref{lem:rrt-heightk} yields that the order $t \ln^{(k)} t$ is the correct exponential decay rate  for the case $t \asymp n$. Moreover,  
together with Lemma \ref{lem:rrt-heightk}, this establishes an  LDP estimate for the $k$-th level set.  This result is of independent interest. 

 \begin{remark} 
  \label{cor:label}
  For any $k \in \mathbb{N}$,  
  \begin{equation}
   \lim_{n \to \infty} \frac{ -\ln \P( H_{n}= k)}{n \ln^{(k)} n } =  1.
  \end{equation}
\end{remark}

Recall that for  $1 \le j \leq m \le n$,  $  \T^{(j)}_{m,n} $ denotes the connected component of $\T_n$ that contains the vertex $v_j$  after removing all edges between vertices in $\{ v_i: i\in [m]\}$. 
 Let  $J_{m,n} $ denote those trees whose sizes are less than $  \theta\frac{n}{m}$: 
\begin{equation}
    J_{m,n}(\theta)\coloneq  \big\{ 1 \leq j \leq m:  |\T_{m,n}^{(j)} | < \theta \frac{n}{m}  \big\} 
\end{equation} 

\begin{lemma}\label{lem:subtree-size-2}
  For any fixed $\theta \in (0,1)$ and $1\ge \gamma > p_{\theta}> 1- e^{-\theta}$. Let $a_{n}$ be an arbitrary increasing sequence tending to infinity.  Then for all $n , m$ large enough and     $ m \le n/a_{n}$,
  \begin{equation}
        \P(   | J_{m,n}(\theta) |  \ge  \gamma  m ) \le  \exp \big(- m\mathsf{D}( \gamma \| p_{\theta} )  \, \big) 
  \end{equation}
   where $\mathsf{D}( \gamma \| p) \coloneq \gamma \ln \frac{\gamma}{p}+	 (1-\gamma) \ln \frac{1-\gamma}{1-p}$ is the binary relative entropy function.
\end{lemma}

\begin{proof}[Proof of Proposition \ref{prop:lowerLDPup} admitting Lemmas \ref{lem:subtree-size-2} and \ref{lem:kth-level-ldp}]
 Define $\mathcal{F}^{\mathrm{t}}_{m}\coloneq \sigma(\mathcal{T}_{m} ,(|\T_{m,n}^{(j)}|)_{j=1}^m )$ and  recall that $\mathsf{F}(n,x)\coloneq \P(H_{n} \leq x)$. Since conditionally on $\mathcal{F}^{\mathrm{t}}_{m}$,  the trees  $(\T_{m,n}^{(j)} , 1 \leq j \leq m)$ are independent random recursive trees on $|\T_{m,n}^{(j)} |$ vertices,   we have 
 \begin{align}
  \P \left(  H_{n} \leq \alpha e\ln n \mid \mathcal{F}^{\mathrm{t}}_{m} \right) &=
  \P \left(  H( |\T_{m,n}^{(j)} | )\leq \alpha e\ln n - \mathrm{dist}(v_{1},v_{j}) , \forall 1 \leq j\leq m  \mid \mathcal{F}^{\mathrm{t}}_{m} \right)  \\
  & = \prod_{j=1}^{m}  \mathsf{F}( |\T_{m,n}^{(j)} |, \alpha e\ln n - \mathrm{dist}(v_{1},v_{j})) .
\end{align} 
Fix a small constant $\theta \in (0,1)$ and take a large integer $k$ such that   $ e\ln(1/\theta) \le k/3$.  Let 
\begin{equation}
   I_{m,n}(\theta,k)\coloneq \left\{  j \in  [m] \setminus J_{m,n}(\theta) : \mathrm{dist}(v_{1},v_{j}) \ge k \right\} .
\end{equation}
Since  $\mathsf{F}( n, x ) $ is decreasing in $n$ and increasing in $x$, we deduce
\begin{equation}\label{eq:cond-bnd}
  \P \left(  H_{n} \leq \alpha e \ln n \right)    
  \leq \E \bigg[ \mathsf{F} \big( \theta\frac{n}{m}, \alpha e\ln n - k \big)^{|I_{m,n}(\theta,k)|}   \bigg]  .
\end{equation} 
 Then  take $ m=m_{n} =  n^{1- \alpha } (\ln  n)^{-3 /(2e)} $   so that $ \theta n/(m) =  \theta  n^{\alpha} (\ln n)^{3 /(2e) }$. 
According to  \eqref{eq-ABF-EH-bound}, we have 
\begin{align}
  \E[ H_{ \theta n/m_{n}  }]
  & = e \ln \left(  \theta n^{\alpha} (\ln  n)^{3  /2e} \right) - \frac{3}{2}  \ln(   \ln n^{\alpha+o(1)} ) + O(1)  \\
  & = \alpha e \ln n - e \ln (1/\theta) - \frac{3}{2} \ln \alpha + O(1), 
\end{align}
The tightness of $(H_{n}- \E[H_{n}])_{n \ge 1}$ yields that if we define $\delta(t)\coloneq  \sup_{n} \P(|H_{n}- \E[H_{n}]|>t) $ then $\delta(t) \to 0$ as $t \to \infty$. Thus provided  $ e\ln(1/\theta) \le k/3$ and $\theta$ is small,  
 \begin{align}
 \mathsf{F} \big( \theta\frac{n}{m_n}, \alpha e\ln n - k \big)   & \le    \P\left( H_{ \theta n/m_{n} } \leq   \E[  H_{ \theta n/m_{n} } ] +O(1) +  e \ln (1/\theta) -k \right) \\
 & \le \delta (k/2).
 \end{align}  
 Plugging this back into \eqref{eq:cond-bnd} we  obtain 
 \begin{equation}
   \P \left(  H_{n} \leq \alpha e \ln n \right)  \le \E \big[ \delta(k/2)^{|I_{m_{n},n}(\theta,k)|}   \big].
 \end{equation}
Observe that $ |I_{m_{n},n}(\theta,k)| \ge  m_n-|J_{m_{n},n}(\theta)|- \sum_{j=0}^{k-1} X_{j}(m_n)$.
It then follows from  Lemmas \ref{lem:kth-level-ldp} and \ref{lem:subtree-size-2} (with $p_{\theta}$ chosen to be $\theta > 1- e^{-\theta}$)  that for large $n$
 \begin{align}
  \P(|I_{m_{n},n}(\theta,k)|  \le  m_{n} /2 ) &\le  \P \biggl(  \sum_{j=0}^{k-1} X_{j}(m_{n}) \ge m_{n} /4 \biggr) + \P(| J_{m_{n},n}(\theta)| \ge m_{n} /4 )   \\
  & \le k\exp(- \tfrac{1}{8} m_{n} \ln^{(k)} m_{n} ) + \exp  ( -m_{n} \, \mathsf{D}( \tfrac{1}{4} \| \theta )   ) .
 \end{align} 
 Finally we conclude that  
 \begin{align}
   \P \left(  H_{n} \leq \alpha e \ln n \right)  \le \delta(k/2)^{m_{n}/2} + k\exp(- \tfrac{1}{8} m_{n} \ln^{(k)} m_{n} ) + \exp (- m_{n} \, \mathsf{D}( \tfrac{1}{4} \| \theta )   ) .
 \end{align}
 Notice that $\mathsf{D}( \frac{1}{4} \| \theta )  \to \infty$ as $\theta \downarrow 0$. The desired result is obtained by taking the negative logarithm of both sides, dividing by $m_{n}$, and sequentially taking the limits as $n \to \infty$, $k \to \infty$, and finally $\theta \downarrow 0$.
\end{proof}

\begin{proof}[Proof of Lemma \ref{lem:kth-level-ldp}] In the following, let $(\eta, \eta_{j}(t): j\ge 1)$ denote a sequence of i.i.d. Poisson processes on $(0,\infty)$ with unit intensity.  

  \underline{\textit{Step 1.}}
Let us analyze the simplest case $k=1$ first. Note that 
  \begin{equation}
    X_{1}(n) = \sum_{j=2}^{n} \ind{ v_{j} \text{ is connected to }  v_{1} } \overset{d}{=}  1+ \sum_{j=2}^{n-1} \xi_{j}
  \end{equation}
  where $(\xi_{j})_{j \ge 2}$ is a sequence of independent Bernoulli random variables  and $\xi_{j} \sim \mathrm{Ber}(1/j)$. Set 
  $\lambda_{j} \coloneq \ln \frac{j}{j-1}$.  Then we can realize $(\xi_{j})$ by setting 
  \begin{equation}
    \xi_{j} \coloneq  \ind{\eta_{j}(\lambda_{j} ) \ge 1} .
  \end{equation} 
In particular, since  $   \sum_{j=2}^{n-1} \eta_{j}(\lambda_{j}) $ has the same distribution as $\eta(\ln(n-1))$, we obtain
\begin{equation}
  \label{eq:X-1-Poi}
  X_{1}(n)  \overset{\mathrm{st}}{\le} 1 + \eta(\ln(n-1)) .
\end{equation} 
We shall apply the following well-known   Chernoff's bound for Poisson distribution (see e.g. \cite[Theorem 5.4]{MU05}). For all $x > \lambda $,
  \begin{equation}\label{Poi-concentration}
    \P( \eta(\lambda) > x ) \le e^{-\lambda} \frac{ (\lambda e )^{x}}{x^{x}} \le \exp  (   -x [\ln x - \ln \lambda -1]  ).
  \end{equation}
Thus we obtain, for all $ \sqrt{n} \le t \le n$,
\begin{equation}
  \label{eq:tailk=1}
\P ( X_{1}(n) \ge t ) \le \exp\big( - (t-1)[ \ln t - \ln^{(2)} n -1 ] \big) .
\end{equation}

\underline{\textit{Step 2.}}
We now return to the general case.
First, we have the following recursive relation: 
\begin{equation}
  X_{k+1} (n )  \overset{\mathrm{d}}{=} \sum_{j=1}^{X_{k}(n)} X^{(j)}_{1} (N_{k}(j)) 
\end{equation}
where $(X^{(j)}_{1}(\cdot) : j \ge 1)$ are i.i.d. copies of  $X_1$ and are independent of $(X_k(n), N_{k}(j): 1\le j \le X_k(n))$.
 To see this, one can explore the tree $\mathcal{T}_{n}$ up to level $k$.  We denote by $T_k(j)$ 
 the subtree rooted at the node with the $j$-th smallest label in level $k$; let $V_k(j)$ be the set of vertices of  $T_k(j)$ and   $N_{k}(j) = | V_k(j) |$.  Let $\mathcal{G}_k$ denote the $\sigma$-field generated by 
  the first $k$ level of $\mathcal{T}_n$ and $V_k(j) :  1 \le j \le X_k(n)$,  $T_k(j)$.    Then from the construction of the random recursive tree, given $\mathcal{G}_k$, $T_k(j)$ are independent random recursive trees on $V_k(j)$.
 Now by applying \eqref{eq:X-1-Poi}, we obtain that conditionally on $\mathcal{G}_k$,
\begin{align}
    X_{k+1} (n )  &\overset{\mathrm{st}}{\le} \sum_{j=1}^{X_{k}(n)} \big[   1+ \eta_{j}( \ln (N_k(j) ) ) \big] = X_{k}(n) + \eta\Big( \sum_{j=1}
    ^{X_{k}(n) }\ln (N_k(j)) \Big) \\
   & \overset{\mathrm{st}}{\le} X_{k}(n) + \eta\Big(  X_{k}(n) \ln \frac{n-1}{X_{k}(n)} \Big) . 
\end{align}
Above we have used Jensen's inequality $\sum_{j=1}^{m} \ln (a_j) \le m \ln (\sum_{j=1}^{m} a_{j}/m )$, as well as  the fact $\sum_{j=1}^{X_k(n)} N_k(j) \le n-1$.  
 In particular, we get for all  $t>0$,
 \begin{equation}\label{eq:recursive-bound}
  \P(   X_{k+1} (n )  > t \mid \mathcal{G}_{k} ) \le \P\Big(  X_{k}(n) + \eta \big(  X_{k}(n) \ln \frac{n-1}{X_{k}(n)} \big)   > t  \mid X_{k}(n) \Big) .
 \end{equation}

 \underline{\textit{Step 3.}}
Fix an arbitrary sequence $(\rho_{j})$ satisfying $1>\rho_{j}> \rho_{j+1}>\rho$ for all $j \ge 1$.   We use induction to show the following claim: For every $k \ge 1$ and for sufficiently large $n$, we have 
 \begin{equation}\label{eq-desired-bound}
  \P( X_{k}(n) \ge t ) \le \exp\Big( - \rho_{k} \, t \ln^{(k)} n \Big) \ \text{ for all } \
 \frac{n}{\ln^{(k)} n} \le t \le n. \end{equation} 
 Then Lemma \ref{lem:kth-level-ldp} follows directly since $ \ln^{(k)} t \sim \ln^{(k)} n$ uniformly in $ \frac{n}{\ln^{(k)} n} \le t \le n$. 

 The initial case $k=1$ follows directly from \eqref{eq:tailk=1}. Assume \eqref{eq-desired-bound} holds for $k$. Choose an arbitrary
 \begin{equation}
  \frac{n}{\ln^{(k+1)} n} \le t \le n   \quad  \text{ and  define } \quad   s \coloneq  t \, \frac{\ln^{(k+1)} n}{\ln^{(k)} n} .
 \end{equation}
 Then $ \frac{n}{\ln^{(k)} n} \le s \le n \frac{\ln^{(k+1)} n}{\ln^{(k)} n} = o(n)$. 
Combining \eqref{eq:recursive-bound} with the fact  that $x \mapsto x \ln \frac{n}{x}$ is increasing in $[1, \frac{n}{e}]$, we obtain 
\begin{equation}\label{eq:induc-kplus1}
   \P(   X_{k+1} (n )  > t  ) \le  \P( X_{k}(n) > s  ) + \P\Big(  s + \eta \big(  s \ln \frac{n}{s} \big)   > t   \Big) .
\end{equation}
By using the induction hypothesis, since $s \ln^{(k)} n  = t \ln^{(k+1)} n $  we obtain
\begin{equation}\label{eq:induc-kplus2}
 \P( X_{k}(n) > s  ) \le \exp(- \rho_{k} s \ln^{(k)} n) = \exp(- \rho_{k} t \ln^{(k+1)} n ).
\end{equation} 
Next we apply \eqref{Poi-concentration} to bound the last term in \eqref{eq:induc-kplus1}, by first checking that $s \ln \frac{n}{s} \ll t-s \Leftrightarrow \ln \frac{n}{s} \ll \frac{t}{s} -1 =\frac{\ln^{(k)} n}{\ln^{(k+1)} n} -1$. Use the lower bound on $s$ we have $\ln \frac{n}{s} \le \ln^{(k+1)} n $. Thus \eqref{Poi-concentration}  yields
\begin{align}
   \P\Big(   \eta \big(  s \ln \frac{n}{s} \big)   > t -s  \Big)  & \le \exp \Big( -(t-s) \big[ \ln \frac{t-s}{s \ln \frac{n}{s}} -1\big]\Big) \\
   &= \exp \Big( -(1-s/t)\, t \big[ \ln \frac{t/s-1}{ \ln \frac{n}{s}} -1\big]\Big)  \\
   &= \exp \Big( -[1+ o_{n}(1)]  t \ln^{(k+1)} n \Big)   . \label{eq:induc-kplus3}
\end{align}
Combining \eqref{eq:induc-kplus2} and \eqref{eq:induc-kplus3} with  \eqref{eq:induc-kplus1}, since $\rho_{k+1}$ is strictly less than $\rho_k$ the desired inequality \eqref{eq-desired-bound} for $k+1$ follows. This completes the proof.
\end{proof}

The proof of Lemma \ref{lem:subtree-size-2} relies on the following concentration inequality for negatively associated Bernoulli random variables:

\begin{lemma}[{\cite[Theorem 1]{IK10}}] \label{generalized Chernoff bound}
Let $\xi_1, \ldots, \xi_n$ be Boolean random variables such that, for some $p \in (0,1)$, we have that, for every subset $S \subseteq[n], \P\left( \wedge_{i \in S} \xi_{i}=1\right) \le p^{|S|}$. Then, for any $ \gamma \in(p ,1]$, 
\[ \P\Big(\sum_{i=1}^n \xi_i \ge \gamma n\Big) \le \exp(-n \mathsf{D}(\gamma \| p) ) . \]
%  where $\mathsf{D}( \gamma \| p) \coloneq \gamma \ln \frac{\gamma}{p}+	 (1-\gamma) \ln \frac{1-\gamma}{1-p}$ is the binary relative entropy function.
\end{lemma}

\begin{proof}[Proof of Lemma \ref{lem:subtree-size-2}] Without loss of generality let us assume $\theta \frac{n}{m}$ is an integer. Otherwise consider $ \lfloor \theta \frac{n}{m} \rfloor$.
We shall apply Lemma \ref{generalized Chernoff bound} and hence first rewrite $
 | J_{m,n}(\theta) |$ as the  sum of indicators:
 \begin{equation}
  | J_{m,n}(\theta) | = \sum_{ j=1}^{m} Y_{j} \ \text{ with } \ Y_{j}\coloneq\ind{|\T_{m,n}^{(j)} | < \theta  \frac{n}{m} } .
 \end{equation}
Recall that $ p_{\theta} \in ( 1-e^{-\theta},1)$.
 We claim that, for sufficiently large $n$ and any $m \le \frac{ n}{a_{n}}$, 
  the negative association condition holds:
 \begin{equation}\label{eq-induction-hypothesis}
  \P(  Y_{j}=1, \forall j \in S ) \leq p_{\theta}^{|S|}  \ \text{ for any } \  S \subset [m].
 \end{equation}
Then applying Lemma \ref{generalized Chernoff bound} the desired result follows.

  We proceed by induction, beginning with showing that for any $1 \leq j \leq m$,
 \begin{equation}\label{eq-induction-initial}
  \P(  Y_{j}=1 )= 1-\P\left(  |\T^{(j)}_{m,n}|  \geq \theta\frac{n}{m}   \right)  \leq p_{\theta} .
 \end{equation}
 It follows from the joint distribution of $ (|\T_{m,n}^{(j)} |: 1 \leq j \leq m) $ in \eqref{eq:subtree-size-distri} and the stars and bars argument that for any  positive integer  $k$ satisfying $ k+ m \leq n$,   we have,
  \begin{align}
   \P( |\T^{(j)}_{m,n}|  \geq 1+k    )  &   =  \frac{ \binom{n-k-1 }{m-1} }{\binom{n-1}{m-1}} 
    =  \prod_{j=1}^{m-1} \frac{n-j-k}{n-j}
     \\
     & = \prod_{j=1}^{m-1} \Big( 1-  \frac{k}{n-j} \Big)  \eqcolon P_{m,n}(k)\label{eq-cdf-T-1-mn}.
  \end{align}
  Taking $k= \theta\frac{n}{m}-1$ we obtain  $\P( |\T^{(j)}_{m,n}|  \geq \theta\frac{n}{m}  ) = [1+o_{n} (1)] (1-\frac{\theta}{m})^{m-1}$, as $m \le \frac{ n}{a_{n}} = o(n)$. Thus   \eqref{eq-induction-initial} follows provided $m$ is large enough.

  Next, assume \eqref{eq-induction-hypothesis} holds for all subsets  $S\subset [m]$ with $|S| \leq r \le m-1$. Due to exchangeability of  $(Y_{j})_{j=1}^{m}$, it suffices to verify for $S'=\{1\} \cup S$ where $1\notin S$ and $|S|=r$.  Actually it is enough to show, for any sequence $(\ell_j)_{j \in S}$ with $1 \leq \ell_j < \theta\frac{n}{m}$, that
  \begin{equation}\label{eq-negative-association}
    \P\big(  |\T^{(1)}_{m,n}| \geq \theta\frac{n}{m}   \mid |\T^{(j)}_{m,n}| = \ell_{j},  j \in S \big) \geq    \P\big(  |\T^{(1)}_{m,n}| \geq \theta\frac{n}{m} \big).
  \end{equation}  
 Given this, it follows immediately that 
\begin{align}
  & \P( |\T^{(j)}_{m,n}| < \theta \frac{n}{m}, j \in S') \\
  &=   \sum_{(\ell_j)}  \P(  |\T^{(1)}_{m,n}|  < \theta  \frac{n}{m}   \mid  |\T^{(j)}_{m,n}| = \ell_{j},  j \in S )\P(  |\T^{(j)}_{m,n}| = \ell_{j},  j \in S)\\
  & \leq      \P(  |\T^{(1)}_{m,n}|    < \theta \frac{n}{m} ) \P( |\T^{(j)}_{m,n}|   < \theta \frac{n}{m}, j \in S)\leq p_{\theta}^{|S'|} .
\end{align}
In the last inequality we have used  the induction hypothesis. 

It remains to prove \eqref{eq-negative-association}.
Notice that,
given $ |\T^{(j)}_{m,n}| = \ell_{j}$ for all $  j \in S  $, 
the conditional distribution of $   |\T^{(1)}_{m,n}| $ is the same as the (unconditional) distribution of  $   |\T^{(1)}_{m-r,n-L}| $ with $L\coloneq\sum_{j \in S}\ell_j \leq r  \frac{\theta n}{m}  $.  Thus, together with   \eqref{eq-cdf-T-1-mn}, we only need to verify that  
for sufficiently large $n$, $m \le n/a_{n}$, and for any $0 \leq t \leq  \theta\frac{n}{m}-1$,
\begin{equation}
 P_{m,n}(t) \le  P_{m-r,n-L}(t).
\end{equation}
We again employ induction on the variable $t$ to prove this inequality. For $t=0$, we trivially have $ P_{m,n}(t) = 1 = P_{m-r,n-L}(t) $. Assume the inequality holds for $t-1$. If we can show 
\begin{equation}\label{eq-stoc-domination}
  \frac{P_{m,n}(t)}{P_{m,n}(t-1)} \leq  \frac{P_{m-r,n-L}(t)}{P_{m-r,n-L}(t-1)}
\end{equation}
for any $t< \theta\frac{n}{m}$, then together with the induction hypothesis, the desired inequality follows. We compute that 
\begin{equation}
  \frac{P_{m,n}(t)}{P_{m,n}(t-1)} =  \frac{ \prod_{j=1}^{m-1}  {n-j-t} }{ \prod_{j=1}^{m-1}  {n-j-t+1} } = 1- \frac{ (m-1) }{n-t} .
\end{equation}
Thus \eqref{eq-stoc-domination} is equivalent to:
\begin{align}
& \frac{ (m-1)}{n-t} \ge \frac{ (m-r-1)}{n-L-t} 
    \Leftrightarrow (n-L-t)  (m-1) \ge (m-1-r) (n-t) \\
   &\ \Leftrightarrow (n-L-t) r  \ge L (m-1-r)   \Leftrightarrow (n-L-t)   \ge \frac{L}{r} (m-1-r) . 
\end{align}
Now, using  $\frac{L}{r} \leq  \frac{\theta n}{m}$ and $t < \theta\frac{n}{m}$,  we obtain  
$\frac{L}{r}(m-1-r) \le \theta n  -(r+1)\frac{\theta n}{m}$ and $n-L-t \ge n  - (r+1)\frac{\theta n}{m}$. This proves \eqref{eq-stoc-domination} and hence  completes the proof. 
\end{proof}

  \section{Upper Large Deviations}

This section is devoted to establishing the following estimate. Then  Theorem~\ref{thm:upTail} becomes an immediate consequence of the proposition stated below, once $\beta$ is replaced by $\beta e$.

  \begin{proposition}\label{prop:uppLDP} 
    Fix $\beta > e$. Define $
\mathsf{J}(\beta) \coloneq  \beta \left( \ln \beta   -1  \right) $. Then  there exist  constants $c,C>0$ such that for all large  $n $,  
\begin{equation}
  \frac{c}{({\ln n})^{3/2+\beta} }   \, n^{  -\mathsf{J}(\beta) } \le    \P \left(  H_{n} \geq \beta \ln n \right) \le  \frac{C}{({\ln n})^{1/2} }    \, n^{  -\mathsf{J}(\beta) }  .
\end{equation} 
   \end{proposition}

   The proof of Proposition \ref{prop:uppLDP}  follows the framework established
in Devroye, Fawzi, and Fraiman \cite{DFF10}, in which they established the law of large numbers
for the height of scaled attachment random recursive trees. 
Throughout this section, we identify the vertices $\{ v_1,\dots,v_n\}$ in a random recursive tree $\mathcal{T}_n$ with the set
$\{0,1,\dots,n-1 \}$. 

We start with   the following simple observation:   
Given a uniformly distributed random variable $U$ in $[0,1]$, the random variable $\lfloor k U \rfloor $ is uniformly distributed
on $\{0,\dots,k-1\}$. Thus,  
we can realize the random recursive tree $\mathcal{T}_n$ as follows:  let $(U_n)_{n \ge 0}$ be a sequence of i.i.d. copies of $U$.  At each step $n$, we introduce a node labeled $n$
and choose its parent   $\pi(n)$  such that 
\[ \pi(n) := \lfloor n U_n  \rfloor \in \{0,1,\dots,n-1\}. \]
For each  $k$, let $\pi^{(k)}$ denote the $k$-fold iteration of the parent function $\pi$. That is, for $k \ge 1$, let 
\begin{equation}
  \pi^{(k)}(n)
  = \left\lfloor \pi^{(k-1)}(n) U_{\pi^{(k-1)}(n)} \right\rfloor
  = \left\lfloor \left\lfloor \cdots \left\lfloor nU_n \right\rfloor U_{\pi(n)} \cdots \right\rfloor U_{\pi^{(k-1)}(n)} \right\rfloor .
\end{equation}  
where  $\pi^{(0)}$ is the identity map.

\begin{lemma}\label{lem:1point-estimate}
 There exists a constant  $C $ such that  for any $ k\ge 1$ and $n \ge 1$ 
  \begin{equation}
  \P( \pi^{(k)}( n) \ge n e^{-k/\beta}  ) \le  \frac{C}{\sqrt{k}} e^{-k[ \frac{1}{\beta} + \ln \beta- 1]} . 
  \end{equation}
 Furthermore, there exists a constant $c >0$ such that  for large $k$ and $n \ge \frac{1}{4} e^{k/\beta}$,
 \begin{equation}
    \P( \pi^{(k)}( n) \ge n e^{-k/\beta}  ) \ge \frac{c}{k^{\beta-1/2}} e^{ - k[ \frac{1}{\beta} + \ln \beta- 1]}  .
 \end{equation} 
\end{lemma}

\begin{proof}
From the definition of $\pi$ we have  the following
  inequality:
\begin{equation}
\label{eq:pi-2-rw}
n \prod_{j=0}^{k-1} U_{ \pi^{(j)} (n)} - k \le \pi^{(k)}(n) 
\le n \prod_{j=0}^{k-1} U_{ \pi^{(j)} (n)} .  
\end{equation}
Observe that the conditional distribution of 
$U_{ \pi^{(j)} (n)}$, given $(\pi^{(i)}(n))_{ 0 \le i \le j}$ 
with $\pi^{(j)}(n) \ge 1$, is still 
the uniform distribution on $[0,1]$. 
Combining this with \eqref{eq:pi-2-rw}, we obtain:
\begin{equation}
\label{eq:pi-k-U-1}
\P ( n U_1 \cdots U_{k} - k \ge n e^{-k/\beta} ) 
\le \P \bigl(\pi^{(k)}(n) \ge n e^{-k/\beta}\bigr) 
\le \P ( n U_1 \cdots U_{k} \ge n e^{-k/\beta} ) . 
\end{equation}  

Let $S_{k}:= - \sum_{j=1}^k\ln U_j$. Note that $\ln (1/U)$ has the standard 
exponential distribution, thus $(S_{k})_{k \ge 1}$ is a 
random walk with standard exponential jump distribution, and $S_k$ follows a Gamma distribution 
with shape parameter $k$ and rate parameter $1$. Then \eqref{eq:pi-k-U-1} is equivalent to: 
\begin{equation}
\label{eq:pi-k-U-15}
\P \Big( S_k \le \frac{k}{\beta} - \ln\big(1+ \frac{k}{n} e^{k/\beta} \big) \Big) 
\le \P \bigl(\pi^{(k)}(n) \ge n e^{-k/\beta}\bigr) 
\le \P ( S_k \le k/\beta ) . 
\end{equation}

For the upper bound, let $z = k/\beta$.
By substituting $u = z - t$, we find: 
\begin{align}
\P ( S_k \le z ) 
&= \int_0^z \frac{u^{k-1}}{(k-1)!} e^{-u} \dif u 
=\int_0^z \frac{(z-t)^{k-1}}{(k-1)!} e^{-(z-t)} \dif t  \\
&= \frac{z^{k-1}}{(k-1)!} e^{-z} \int_0^z (1 - \frac{t}{z})^{k-1} e^t \dif t \lesssim \frac{z^{k-1}}{(k-1)!} e^{-z} .
\end{align}
In the last inequality, we have used the inequality $1 - u \le e^{-u}$,  and obtained 
$\int_0^z (1 - \frac{t}{z})^{k-1} e^t \dif t  
\le \int_0^\infty e^{-t (\frac{k-1}{k}\beta - 1)} \dif t < \infty  $ since $\beta > e$.  Using Stirling's approximation   
$(k-1)! = \Theta( k^{-1/2} (k/e)^k )$, 
and plugging  $z = k/\beta$, we obtain:
\begin{equation}
  \label{eq:gamma-density-k-beta}
\P ( S_k \le k/\beta ) 
\lesssim  \frac{z^{k-1}}{(k-1)!} e^{-z} = \Theta(  \frac{1}{\sqrt{k}} e^{-k[ \frac{1}{\beta} + \ln \beta - 1] } ).
\end{equation}

For the lower bound, let 
$w = k/\beta - \ln(1+4k)$. 
By assumption, we have
$\ln ( 1+ e^{k/\beta} k /n) \le \ln(1+4k)$. Thus,
\begin{align}
\P \bigl(\pi^{(k)}(n) \ge n e^{-k/\beta}\bigr) 
&\ge \P ( S_k \in [w-1, w] )  \\
&
= \int_{w-1}^w \frac{t^{k-1}}{(k-1)!} e^{-t} \dif t \ge \frac{(w-1)^{k-1}}{(k-1)!} e^{-w} .
\end{align}
Recall that $z = k/\beta$.
We expand the lower bound term:
\begin{equation}
   \frac{(w-1)^{k-1}}{(k-1)!} e^{-w} 
= \frac{z^{k-1}}{(k-1)!} e^{-z} 
\Big(1 - \frac{1+\ln(1+4k)}{k/\beta}\Big)^{k-1} (1+ 4k).
\end{equation}
For large $k$, the term
$\big(1 - \frac{1+\ln(1+4k)}{k/\beta}\big)^{k-1} $ is of order $\Theta( k^{-\beta} )$. Finally, combining these polynomial terms yields:
\begin{equation}
\P \bigl(\pi^{(k)}(n) \ge n e^{-k/\beta}\bigr) 
\gtrsim  k^{\frac{1}{2} - \beta} e^{-k[ \frac{1}{\beta} + \ln \beta - 1]} .
\end{equation}
This completes the proof.
\end{proof}

 \begin{proof}[Proof of the upper bound in Proposition \ref{prop:uppLDP} ] 
 Observe that, from the definition of the mapping $\pi$, it follows that $  \mathrm{dist}(0,t)=\min\{ k \ge 0: \pi^{(k)}(t) = 0\}$. Hence we have:
\begin{equation} 
  \label{eq:dist-to-pi}
  \{ \mathrm{dist}(0,t) \ge k  \} = \{  \pi^{(k)}(t)  >0 \}  = \{  \pi^{(k)}(t)  \ge 1 \} .  
\end{equation}
 Let $b_n := \lceil  \beta \ln n \rceil $. By applying the union bound and using the fact  $ t e^{-b_n / \beta}< 1$ for all $1 \le t <n$, we get 
  \begin{equation}
        \P \left(  H_{n} \geq \beta \ln n \right) \le \sum_{t=1}^{n-1} \P  \bigl( \pi^{(b_n)}(t) \ge 1  \bigr) \le \sum_{t=1}^{n-1} \P  \Bigl(  \pi^{(b_n)}(t) \ge t e^{-b_n / \beta} \Bigr) .
  \end{equation}
Lemma \ref{lem:1point-estimate} then implies:
    \[ \P \left(  H_{n} \geq \beta \ln n \right) \lesssim n \, \frac{1}{\sqrt{\ln n}} e^{- \beta \ln n [ \frac{1}{\beta} + \ln \beta- 1]  }=  \frac{1}{\sqrt{\ln n}} n^{- \mathsf{J}(\beta) } ,  \]
    as desired. This completes the proof.
 \end{proof}

 The proof of the lower bound relies on the second moment method. We define a set of "good" vertices as follows: let 
\[ G_{\beta}(k) = \bigl\{  x \ge 1  :  \pi^{(j)}(x) \ge x e^{- j/\beta  } \text{ for all } 1 \le j \le k  \bigr\} . \]

\begin{lemma}
  \label{lem:1point}
There exists a constant $c>0$ such that for all large $k$ and $x \ge \frac{1}{4} e^{k/\beta}$,
  \begin{equation}
     \P( x \in G_{\beta}(k) )  \ge \frac{c}{k^{1/2+\beta}} e^{-k[ \frac{1}{\beta} + \ln \beta- 1]} .   
  \end{equation}
\end{lemma}

  \begin{proof}[Proof of Lemma \ref{lem:1point}]
Following the argument in Lemma \ref{lem:1point-estimate}, with
$S_j:=-\sum_{i=1}^{j}\ln U_i$, we have
\begin{align}
 \P(x\in G_{\beta}(k))
 &\ge \P\Big(S_j\le \frac{j}{\beta}
 -\ln\big(1+\frac{j}{x}e^{j/\beta}\big)
 \text{ for all }1\le j\le k\Big).
\end{align}
Since $x\ge \frac14 e^{k/\beta}$, the right-hand side is bounded from below by
\begin{equation}\label{eq:good-path-barrier}
 \P\left(S_j\le g_k(j)\text{ for all }1\le j\le k\right) \ , \ 
 \text{ where } \
 g_k(t)\coloneq \frac{t}{\beta}
 -\ln\big(1+4t e^{-(k-t)/\beta}\big).
\end{equation}

We claim that, for all sufficiently large $k$,
\begin{equation}\label{eq:good-path-linear-barrier}
 g_k(t)\ge \frac{t}{k}g_k(k)\quad\text{for all }1\le t\le k.
\end{equation} 
Let $w_k\coloneq g_k(k)=k/\beta-\ln(1+4k)$. Using  
\eqref{eq:good-path-linear-barrier}, we have   
\begin{equation} \label{eq-barrier}
    \Bigl\{ S_k\in[w_k-1,w_k],\
   S_j\le\frac{j}{k}S_k \,,\ \forall \ 1\le j\le k    \Bigr\} \subset \bigl\{ S_j\le g_k(j) \,,\ \forall \ 1\le j\le k \bigr\} .
\end{equation}
Feller's cycle lemma argument (see e.g. \cite[Lemma 5]{DR95}) implies 
\begin{equation}
  \label{eq:cycle-lemma}
  \P  \Bigl(  
 S_j\le\frac{j}{k}S_k \,,\ \forall \ 1\le j\le k  \, \Big| \,  S_k = s  \Bigr) = \frac{1}{k}. 
\end{equation} 
Now, combining \eqref{eq-barrier} with \eqref{eq:cycle-lemma}, we get 
\begin{align}
 \P(x\in G_{\beta}(k)) &\ge \P \Bigl(  S_k\in[w_k-1,w_k],\
 S_j\le\frac{j}{k}S_k \,,\ \forall \ 1\le j\le k    \Bigr) \\
 &\ge \frac{1}{k}\P\left(S_k\in[w_k-1,w_k]\right) \ge c k^{-1/2-\beta}
 e^{-k[\frac{1}{\beta}+\ln\beta-1]}, \label{eq:good-path-ballot}
\end{align} 
as required. 

It remains to prove  
\eqref{eq:good-path-linear-barrier}.
Indeed, set
\[
 q_k(t)\coloneq \frac{1}{t}\ln\big(1+4t e^{-(k-t)/\beta}\big).  
\]
It suffices to show $q_k$ is increasing on $[1,k]$, i.e., $q'_k(t)  \ge 0$ on $[1,k]$. 
Writing $y=y_k(t)=4t e^{-(k-t)/\beta}$,  we have 
$ t^2 q_k'(t)
  =\frac{y(1+t/\beta)}{1+y}-\ln(1+y) $.  
For all sufficiently large $k$ and all $t\in[1,k]$, we have $\ln(1+y)\le t/\beta$. Indeed, this inequality is equivalent to 
\begin{equation}\label{eq:good-path-elementary-bound}
 e^{-t/\beta}+4t e^{-k/\beta}\le 1.
\end{equation}
The left-hand side is convex in $t$, and the
inequality holds at both endpoints for all sufficiently large $k$.  Consequently,
\begin{align}
 t^2 q_k'(t) 
  \ge \frac{y(1+\ln(1+y))}{1+y}-\ln(1+y)
 =\frac{y-\ln(1+y)}{1+y}\ge0.
\end{align}
This completes the proof.
   \end{proof}

We shall also use the following lemma, which is essentially 
{\cite[Lemma 4]{DFF10}}. For the reader's convenience, we recall its proof 
in the appendix, adapted to the notation of the present paper.

\begin{lemma}[{\cite[Lemma 4]{DFF10}}]\label{2point-estimate} Let $x<y$ be two    positive integers. Then  
  \begin{equation}
    \P( x, y \in   G_{\beta}(k)    ) \le  \sum_{\ell=0}^{k-1} \P(x \in G_{\beta}(k) )\P(y \in G_{\beta}(\ell) ) \frac{(k+1)}{y e^{- \ell/\beta}} + \P( x \in G_{\beta}(k)) \P( y  \in G_{\beta}(k)  ).
  \end{equation} 
\end{lemma}

Now, we are ready to complete the proof of Proposition \ref{prop:uppLDP}.

 \begin{proof}[Proof of the lower bound in Proposition \ref{prop:uppLDP}]
Recall that $b_n := \lceil \beta \ln n \rceil$.
Observe that 
if $x \in G_{\beta}(b_n)$,
then we must have
$\pi^{(b_n)}(x) \ge \frac{n}{2} e^{-b_n / \beta} > 0$.
By \eqref{eq:dist-to-pi}, this condition implies
that $\mathrm{dist}(0,x) \ge b_n$.
Now, let us define
\[ \Sigma_n := |\{ n/2 \le x \le n-1 : x \in G_{\beta}(b_n)\}|. \]
Applying the Cauchy--Schwarz inequality,
we obtain:
\begin{equation}\label{eq:CS-ineq}
\P(H_n \ge \beta \ln n) \ge \P(\Sigma_n \ge 1)
\ge \frac{(\E[\Sigma_n])^2}{\E[\Sigma_n^2]}.
\end{equation}
First, from Lemma \ref{lem:1point}, we have:
\begin{equation}\label{1st-mt-low}
\E[\Sigma_n] \ge  \sum_{n/2 \le x < n}
\P (  x \in G_{\beta}(b_n))
\gtrsim \frac{1}{(\ln n)^{1/2+\beta}} n^{-\mathsf{J}(\beta)}.
\end{equation}
Next, applying  
Lemma \ref{2point-estimate} yields  
\begin{equation}\label{2nd-mt-1}
\E[\Sigma_n^2] \le \E[\Sigma_n] + (\E[\Sigma_n])^2
+ \sum_{n/2 \le y,x < n} \sum_{\ell=1}^{b_n-1}
\P(x \in G_{\beta}(b_n)) \P(y \in G_{\beta}(\ell))
\frac{b_n+1}{y e^{-\ell/\beta}}.
\end{equation}
Furthermore, for all $n/2 \le y < n$,
Lemma \ref{lem:1point-estimate} provides  
\[ \P(y \in G_{\beta}(\ell)) \le
\P(\pi^{(\ell)}(y) \ge y e^{-\ell/\beta})
\le e^{-\ell[1/\beta+\ln \beta - 1   ]}. \]
Substituting this into the upper bound
for $\E[\Sigma_n^2]$, and noting that
$\sum_{\ell \ge 1} e^{-\ell[\ln \beta - 1]}$
is summable, we conclude 
\begin{equation}\label{2nd-mt-2}
\E[\Sigma_n^2] \lesssim (\E[\Sigma_n])^2
+ (b_n + 1) \sum_{n/2 \le x < n}
\P(x \in G_{\beta}(b_n))
\lesssim   (\ln n) \E[\Sigma_n].
\end{equation}
Combining \eqref{eq:CS-ineq}, \eqref{1st-mt-low},
\eqref{2nd-mt-1}, and \eqref{2nd-mt-2}
gives the desired lower bound.
\end{proof}

\appendix 

\section{Proof of Lemma \ref{2point-estimate}}

  \begin{proof}[Proof of Lemma \ref{2point-estimate}] 
   Let $\Pi_{k}(x)= ( \pi^{(j)}(x) )_{j=0}^{k}$ denote the unique path of length $k$   from node $x$ to the root $0$. 
  Let  $\tau+ 1$ denote the collision step, at which the path starting at $y$  first intersects  the path from $x$. Formally, we define: 
  \[  \tau \coloneq \min\{ 0 \le \ell  \le k-1: \pi^{(\ell+1)}(y) \in \Pi_{k}(x) \}  \ \text{ if } \ \Pi_{k}(x)  \cap \Pi_{k}(y) \neq \emptyset , \]
and $\tau \coloneq \infty$ otherwise.  We write $
P \in G_{\beta}(k)$ if $P$ is a possible realization of $\Pi_{k}(x)$ that guarantees $x \in G_{\beta}(k)$. Then 
for any $0 \le \ell \le k-1  $ we have 
\begin{align}
   & \P  (\tau = \ell,  x, y \in   G_{\beta}(k) )  = \sum_{P \in G_{\beta}(k)} 
   \P ( \Pi_k(x) = P )  \P( \tau = \ell, y \in   G_{\beta}(k)  \mid \Pi_k(x) = P ) \\
   & \le \sum_{P \in G_{\beta}(k)} 
   \P ( \Pi_k(x) = P )  \sum_{ \substack{ z \notin P \\ z \ge y e^{-\ell/\beta} }}  \P \bigg( \begin{array}{l}
   \Pi_{\ell}(y) \in G_{\beta}(\ell),
      \Pi_{\ell}(y) \cap P = \emptyset \\
      \pi^{(\ell)}(y)=z ,  \lfloor z U_{z} \rfloor \in P 
   \end{array} \mid \Pi_k(x) = P \bigg)
\end{align}
 Observe that by our construction, conditioned on the event  $\{\Pi_k(x) = P\}$ happen or not,  the probability of the event $\{\Pi_{\ell}(y) \cap P = \emptyset, \pi^{(\ell)}(y)=z\}$ are the same. Note that $U_z$ is independent of $\{\Pi_{\ell}(y) \cap P = \emptyset \} $. In summation we obtain 
\begin{align}
 & \P (
        \Pi_{\ell}(y) \cap P = \emptyset, \pi^{(\ell)}(y)=z ,  \lfloor z U_{z} \rfloor \in P \mid \Pi_k(x) = P ) \\
        &=  \P (
        \Pi_{\ell}(y) \cap P = \emptyset, \pi^{(\ell)}(y)=z ,  \lfloor z U_{z} \rfloor \in P  )\\
        & \le \P (
         \pi^{(\ell)}(y)=z ) \,  \P(\lfloor z U_{z} \rfloor \in P  ) .
\end{align}
Furthermore, since $z \ge y e^{-\ell/\beta}$, and $|P| \le k+1$ we have  
\[ \P(\lfloor z U_{z} \rfloor \in P  ) \le   \frac{ k+1}{ y e^{-\ell/\beta}}  .\]
All together, summing over $\ell$ yields the bound for
non-empty intersections: 
\begin{align}
&  \P  (  x, y \in   G_{\beta}(k) ,  \Pi_k (y) \cap \Pi_k (x) \neq \emptyset ) \\
  & \le  \sum_{\ell=0}^{k-1} \sum_{P \in G_{\beta}(k)}  \P ( \Pi_k(x) = P )     \frac{ k+1}{ y e^{-\ell/\beta}} 
  \sum_{   z \notin P , z \ge y e^{-\ell/\beta} } \P (  
   \Pi_{\ell}(y) \in G_{\beta}(\ell) ,
      \pi^{(\ell)}(y)=z    
 )   \\
  & = \sum_{\ell=1}^{k-1} \P(x \in G_{\beta}(k) )\P(y \in G_{\beta}(\ell) ) \frac{(k+1)}{y e^{- \ell/\beta}}  .
\end{align}

Finally, for the case $\tau=\infty$, repeating the above argument one get 
\begin{align}
  \P  (  x, y \in   G_{\beta}(k) ,  \Pi_k (y) \cap \Pi_k (x) =\emptyset ) \le \P( x \in G_{\beta}(k)) \P( y  \in G_{\beta}(k)  ).
\end{align}
Combining these cases completes the proof.   \end{proof}

  %  \heng{Upper bound is a simple union bound, and the lower bound shoule be also easy because we have the BRW embeding ?? Otherwise we have a second moment argument}

  % \section{Spine decomposition}

%%%%%%%%%%%%%%%%%%%%%%%%%%%%%%%%%%%%%%%%%%%%%%%%%%%%%%%%%%%%%%%%%%%%%%%%%%%%%%%%%%
%% Acknowledgement %%%%%%%%%%%%%%%%%%%%%
%%%%%%%%%%%%%%%%%%%%%%%%%%%%%%%%%%

\section*{Acknowledgement}

The authors express their gratitude to Prof. Jian Song for his generous hospitality and
support during their research visit at the Research Center for Mathematics and Interdisciplinary Sciences, Shandong
University in the Summer of 2025, during which the main idea of this work was conceived.
X.C. is supported by National Key R\&D Program of
China (No. 2022YFA1006500). H.M. is supported in part from a Lady
Davis Fellowship at the Technion.

%%%%%%%%%%%%%%%%%%%%%%%%%%%%%%%%%%
%% REFERENCE %%%%%%%%%%%%%%%%%%%%%
%%%%%%%%%%%%%%%%%%%%%%%%%%%%%%%%%%
 
 \bibliographystyle{alpha}
 \bibliography{biblio}

 \end{document}